\documentclass[12pt,reqno]{amsart}

\usepackage[margin=1.25in]{geometry}
\usepackage[UKenglish]{babel} 
\usepackage[T1]{fontenc}
\usepackage[ansinew]{inputenc}
\usepackage{graphicx} 
\usepackage{lmodern} 
\usepackage{enumerate}
\usepackage{amsmath}
\usepackage{amsthm}
\usepackage{amsfonts}
\usepackage{amssymb}
\usepackage{amscd}
\usepackage{bm}
\usepackage{xcolor}
\usepackage{cleveref}
\usepackage[numbers,sort]{natbib}
\usepackage[colorinlistoftodos]{todonotes}
\usepackage{tikz-cd} 
\usepackage[shortlabels]{enumitem}
\newcommand{\ldb}{\{\!\!\{}
\newcommand{\rdb}{\}\!\!\}}
\newcommand{\PP}{\mathbb{P}} 
\newcommand{\CC}{\mathbb{C}}
\newcommand{\BB}{\mathcal{B}}

\newcommand{\RR}{\mathbb{R}}
\newcommand{\R}{\mathcal{R}}
\newcommand{\C}{\mathcal{C}}
\newcommand{\TT}{\mathbb{T}}
\newcommand{\QQ}{\mathbb{Q}}
\newcommand{\KK}{\mathbb{K}}
\newcommand{\conv}{\operatorname{conv}}

\newcommand{\id}{\operatorname{id}}

\newcommand{\trop}{\operatorname{trop}}
\newcommand{\Sym}{\operatorname{Sym}}
\newcommand{\init}{\operatorname{in}}
\newcommand{\subdiv}{\mathcal{S}}
\newcommand{\vertices}{\operatorname{Vert}}

\newcommand{\fvm}{{\bm{\mu}}}
\newcommand{\fM}{\textbf{M}}

\newcommand{\fL}{\textbf{L}}
\newcommand{\fd}{\textbf{d}}
\newcommand{\fn}{\textbf{n}}
\newcommand{\fw}{\textbf{w}}
\newcommand{\fB}{\textbf{B}}
\newcommand{\fS}{\textbf{S}}

\newcommand{\tnn}{\textsl{tnn}}

\newcommand{\BP}{Q}
\newcommand{\tBP}{\widetilde{Q}}
\newcommand{\BI}{\operatorname{BI}}
\newcommand{\tBI}{\widetilde{\operatorname{BI}}}

\newcommand{\bemph}[1]{\emph{#1}}

\newcommand{\TP}{\mathbb{TP}^{n-1}} 

\newcommand{\Dr}{\operatorname{Dr}} 
\newcommand{\FlDr}{\operatorname{FlDr}}
\newcommand{\Fl}{\operatorname{Fl}} 
 
\newcommand{\TFl}{\operatorname{TFl}} 

\theoremstyle{plain}
\newtheorem{question}{Question}
\newtheorem{proposition}{Proposition}[section]
\newtheorem{theorem}[proposition]{Theorem}
\newtheorem{lemma}[proposition]{Lemma}
\newtheorem{conjecture}[proposition]{Conjecture}
\newtheorem{corollary}[proposition]{Corollary}

\theoremstyle{definition}
\newtheorem{definition}[proposition]{Definition}
\newtheorem{example}[proposition]{Example}

\theoremstyle{remark}
\newtheorem{remark}[proposition]{Remark}

\numberwithin{equation}{section}

\begin{document}

\pagestyle{plain}

\title{Positivity for partial tropical flag varieties}
\author{Jorge Alberto Olarte}

\address
{
CUNEF Universidad, Departamento de m\'etodos cuantitativos
}
\email{jorge.olarte@cunef.edu}

\begin{abstract}
We study positivity notions for the tropicalization of type $\textsc{A}$ flag varieties and the flag Dressian.
We focus on the hollow case, where we have one constituent of rank 1 and another of corank 1.
We characterize the three different notions that exist for this case in terms of Pl\"ucker coordinates.
These notions are the tropicalization of the totally non-negative flag variety, Bruhat subdivisions and the non-negative flag Dressian.
We show the first example where the first two concepts above do not coincide.
In this case the non-negative flag Dressian equals the tropicalization of the Pl\"ucker non-negative flag variety and is equivalent to flag positroid subdivisions. 
Using these characterizations, we provide similar conditions on the Pl\"ucker coordinates for flag varieties of arbitrary rank that are necessary for each of these positivity notions, and conjecture them to be sufficient.
In particular we show that regular flag positroid subdivisions always come from the non-negative Dressian.
Along the way, we show that the set of Bruhat interval polytopes equals the set of twisted Bruhat interval polytopes and we introduce valuated flag gammoids.
\end{abstract}

\maketitle

\date{}

\setcounter{tocdepth}{1}
\tableofcontents



\section{Introduction}
The flag variety $\Fl(\fd,\KK^n)$ of rank $\fd=(d_1,\dots,d_s)$ over a field $\KK$ consists of all flags of linear spaces $L_1\subset \dots L_s \subseteq \KK^n$ of dimensions $d_1<\dots<d_s$ respectively.
Using the Pl\"ucker embedding we have that $\Fl(\fd,\KK^n)$ is a projective variety satisfying
\begin{equation}
\sum_{t_i\in T\backslash S}^s (-1)^ix_{St_i}x_{T\backslash t_i} = 0
\label{eq:incidence}
\end{equation}
for every $j\in [s-1]$ and pair of subsets $(S,T)$ where $S\in {[n] \choose d_j-1}$ and $T$ is either in ${[n] \choose d_j+1}$ (Pl\"ucker relations) or $T$ is in ${[n] \choose d_{j+1}+1}$ (incidence relations), where $t_1<\dots< t_s$ are the elements of $T\backslash S$ in order. 

Positivity in the flag variety and its tropicalization has recently gathered a lot of attention \cite{arkani2017positive,bloch2022two,boretsky,BEW,JLLO,KodamaWilliams:2015,TsukermanWilliams}.
There are two well studied notions of non-negativity for the flag variety over the reals.
The first, notion simply asks that all all the Pl\"ucker coordinates are non-negative.
This is called Pl\"ucker non-negative in \cite{bloch2022two} or naive non-negative in \cite{arkani2017positive}. 
This notion comes as a generalization of poistroids and positroid polytopes which have been well studied \cite{postnikov,ardila2016positroids,lam2020polypositroids}

The second notion, was originally introduced by Lustzig \cite{lusztig1994total,Lusztig:1998}.
For the complete flag variety, this notion is equivalent to Pl\"ucker non-negativity \cite{boretsky, bloch2022two},
and therefor this second notion is equivalent to asking whether the flag can be extended to a complete flag which is Pl\"ucker non-negative.
The part of the flag variety satisfying this condition is called the totally \bemph{non-negative} (or $\tnn$) \emph{flag variety}.

The support vector of a flag in $\Fl(\fd,\KK^n)$, that is, the set of $B$ such that $x_B\ne 0$ are the bases of a flag matroid $(M_1,\dots,M_s)$ of rank $\fd$ (see \Cref{sub:fm} for a precise definition of flag matroid).
A flag matroid $\fM$ can be encoded in a polytope $P_\fM$ which is the Minkowski sum $P_{M_1}+\dots+P_{M_s}$ of the polytopes of the matroid constituents. 
The polytopes of flag matroids in the $\tnn$ flag variety were shown to be of as special kind, namely Bruhat (interval) polytopes \cite{KodamaWilliams:2015}.
These are polytopes associated to intervals in the poset of permutations given by the Bruhat order which were further studied in \cite{TsukermanWilliams}. 
They are in correspondence with the Richardson varieties that were shown to subdivide the $\tnn$ flag variety \cite{rietsch1998total} and are also indexed by Bruhat intervals.
Parametrizations of Richardson varieties were first provided by Marsh and Rietsch in \cite{MR}.  
There have been two different definitions of Bruhat polytopes (see \Cref{def:tbru,def:bru}) which were known to give the same class of polytopes for complete rank, but this was not previously known for general rank. This is our first result:

\begin{theorem}[\Cref{thm:twisted}]
\label{thm:introbru}
The set of twisted Bruhat polytopes is the same as the set of untwisted Bruhat polytopes.
\end{theorem}

The tropical flag variety $\TFl(\fd,n)$ was introduced in \cite{BEZ} as the tropicalization of the flag variety via its Pl\"ucker embedding.
The corresponding prevariety $\Fl\Dr(\fd,n)$ is called flag Dressian and it is defined by 
\begin{equation}
\bigoplus_{i\in T\backslash S} \fvm(Si) \odot \fvm(T\backslash i) \text{ achieves the minimum at least twice.}
\label{eq:troplucker}
\end{equation}
for all pairs of sets $(S,T)$ as in (\ref{eq:incidence}). 
For $T\in {[n] \choose d_j+1}$ these are the tropical Pl\"ucker relations introduced in \cite{SS} for the tropical Grassmannian.
For $T\in {[n] \choose d_{j+1}+1}$ these are the tropical incidence relations which were first shown by Haque \cite{Haque} to characterize containment of tropical linear spaces.
Much like Dressians characterize regular subdivisions of matroid polytopes into matroid polytopes \cite{Speyer:2008}, the flag Dressian characterizes coherent mixed subdivisions of flag matroid polytopes into flag matroid polytopes \cite{BEZ}.

The study of positivity of the tropical Grassmannian  was initiated by Speyer and Williams \cite{SpeyerWilliams:2005} where it was parametrized using a tropicalized version of the paramterizations used by Postnikov for positroid cells \cite{postnikov}. 
It was simultaneously proven in \cite{ArkaniHamedLamSpradlin:2021} and \cite{SpeyerWilliams:2021} that the positive tropical Grassmannian is determined by the three-term positive relations and that it contained all valuated matroids inducing positroid subdivisions. 
Positroid subdivisions are of particular interest as they induce subvidisions of the amplithuedron introudced by Arkani-Hamed and Trnka \cite{arkani2014amplituhedron} for the study of scattering amplitudes in the $\mathcal{N}=4$ SYM model \cite{arkani2017positive,ArkaniHamedLamSpradlin:2021,lukowski2020positive}. 

The positivity notions of the flag variety can be taken to the tropical context, by looking at the image under the valuation map of the flag variety defined over the Puiseux series. 
In this paper, we consider the following notions of non-negativity for valuated flag matroids:

\begin{enumerate}[(a)]
	\item The tropicalization of the $\tnn$ flag variety, $\TFl^\tnn(\fd,n)$. These are the valuated flag matroids that can be realized by a flag of linear spaces coming from a matrix where all the minors are positive.
	We call such valuated flag matroids \bemph{totally non-negative}. \label{enum:tnn}
	\item Valuated flag matroids inducing subdivisions consisting exclusively of Bruhat polytopes. We call such subdivisions \bemph{Bruhat subdivisions}. \label{enum:bru}
	\item The tropicalization of the Pl\"ucker non-negative flag variety, $\TFl^{\ge 0}(\fd,n)$. These are the flag matroids that can be realized by a flag of linear spaces whose $d_i\times d_i$ minors using the first $d_i$ rows are positive. Such flag matroids were \bemph{called flag positroids} in \cite{BEW}. \label{enum:pl+}
	\item Valuated flag matroids inducing subdivisions consisting exclusively of flag positroid polytopes. We call such subdivisions \bemph{flag postroid subdivisions}. \label{enum:positroid}
	\item The \bemph{non-negative flag Dressian} $\FlDr^{\ge0}(\fd,n)$, consisting of all flag valuated matroids $\fvm$ satisfying
\begin{equation}
\bigoplus_{\substack{t_i \in T\backslash S \\ i \text{ odd}}} \fvm(Si) \odot \fvm(T\backslash i)   = \bigoplus_{\substack{t_i \in T\backslash S \\ i \text{ even}}} \fvm(Si) \odot \fvm(T\backslash i) 
\label{eq:oriented}
\end{equation}	
	for all $(S,T)$ as in (\ref{eq:incidence}). This definition can be rephrased as flag matroids over the hyperfield of signed tropical numbers \cite{JO} which are positively signed.  \label{enum:dress}
\end{enumerate}

All of these notions of positivity were shown to be the same for the Grassmannian \cite{SpeyerWilliams:2021,ArkaniHamedLamSpradlin:2021}.
Then it was proved in \cite{JLLO} that these notions agree for the flag variety with full support (i.e. when the flag matroid polytope is the permutahedron).
Both of this cases were then generalized for flags of consecutive rank in \cite{BEW} without necessarily full support. 
The main goal of this paper is to look at what happens to these notions whenever there are gaps $\fd$.

It is immediate that \ref{enum:tnn} is stronger than \ref{enum:pl+}, \ref{enum:bru} is stronger than \ref{enum:positroid} and \ref{enum:pl+} is stronger than both \ref{enum:positroid} and \ref{enum:dress}. 
It follows from \cite{KodamaWilliams:2015} (see \Cref{prop:richard}) that \ref{enum:tnn} is stronger than \ref{enum:bru}.
An example is given in \cite[Example 4.6]{BEW} showing that \ref{enum:pl+} does not imply \ref{enum:tnn} or \ref{enum:bru} for non-consecutive flags.
We provide the first example showing that \ref{enum:tnn} is stronger than \ref{enum:bru} for non-consecutive flags (see \Cref{ex:main}).
Moreover, this example shows that total positivity can can not be read from the matroid subdivision.
It is asked in \cite[Question 1.6]{BEW} whether \ref{enum:pl+} and \ref{enum:dress} are equivalent for unvaluated matroids. 
This can be thought of as a flag generalization of the fact that positivity implies realizability for oriented matroids \cite{ardila2017positively}.

To study non consecutive flags, it makes sense to focus on the extremal case of rank $(1,n-1)$.
For Dressians, all the geometric information is concentrated in the octahedron because it is the matroid polytope of $U_{2,4}$, the smallest matroid where where a non-trivial tropical Pl\"ucker relation appears.
Similarly hexagons are the polytope of $(U_{1,3}, U_{2,3})$, the smallest flag matroid containing a non-trivial tropical incidence relation, so most of the geomteric information is contained in them (although the subdivision is not determined by them, see \cite[Figure 3]{JLLO}).
So the idea is that copies of the polytope of $(U_{1,n},U_{n-1,n})$, the smallest flag matroid where the equations above have $n$ terms, contain enough information to determine all these positivity concepts, just like in the consecutive rank case.
We use the adjective \bemph{holllow} for matroids of rank $\fd=(1,n-1)$.

In the hollow case, there is only one incidence relation (and no non-trivial Pl\"ucker relations).
To simplify notations, we write
\[
\lambda_i := \fvm(i)\odot\fvm([n]\backslash i)
\]
for the terms that appear in the incidence relation.
Formally, we have a vector $\lambda\in \TP$ in tropical projective space.
It is invariant under translation of tropical linear spaces (scaling of coordinates) as it is homogenous. 
Geometrically, we can think of the lambda values as the relative heights of the labeled points in the mixed subdivision induced by $\fvm$ that land in the center of $\Delta(1,n-1;n)$.
For what lies below, we also let $\lambda_0 = \lambda_{n+1} = \infty$.

Our next main result, is completely characterizing all positivity notions for hollow rank in terms of the quadratic terms of the Pl\"ucker relations (see \Cref{fig:FlDr}).
In particular we answer \cite[Question 1.6]{BEW} positively for this case.

\begin{theorem}[\Cref{thm:bruh,thm:tnnh,thm:posh}]
\label{thm:hollow}
Let $\fvm\in \FlDr(1,n-1;n)$ be a flag valuated matroid of hollow rank.
\begin{enumerate}[1.]
	\item $\fvm$ is totally non-negative if and only if	
	 \begin{equation}
\lambda_i \ge \lambda_{i-1} \oplus \lambda_{i+1}
	 \label{eq:tnnh}
	 \end{equation}	
	for every $i\in [n]$.
	\item $\fvm$ induces a Bruhat subdivision if and only if	
	 \begin{equation}
	\lambda_i = \lambda_{i-1} \oplus \lambda_{i+1}
	 \label{eq:bruh}
	 \end{equation}
	for every $i$ such that $\lambda_i = \lambda_1\oplus\dots\oplus\lambda_j$.
	\item The following are equivalent:
	\begin{enumerate}[i.]
		\item $\fvm$ is a valuated flag positroid
		\item $\fvm$ induces a flag positroid subdivison 
		\item $\fvm\in \FlDr^{\ge 0}(1,n-1;n)$, i.e. 
			\begin{equation}
	\bigoplus_{i \text{ odd}} \lambda_i   = \bigoplus_{i \text{ even}} \lambda_i 
	 \label{eq:posh}
	 \end{equation}
	\end{enumerate}
\end{enumerate}
\end{theorem}

In the process of proving \Cref{thm:hollow} we introduce valuated flag gammoids (see \Cref{thm:gammoid} and \Cref{def:gammoid}).
They generalize valuated gammoids introduced in \cite{FO}, which are in turn a valuated generalization of the matroid class introduced by Mason in \cite{MasonGammoids}.
We use them to simplify the Marsh-Rietsch parametrization of Richardson cells, which was first given a graphic interpretation by Boretsky \cite{boretsky}. 

As our results are all related to the quadratic monomials that appear in (\ref{eq:troplucker}), we introduce the following notation:
\begin{definition}
\label{def:lambda}
Given a valuated flag matroid $\fvm$ of rank $\fd = (d_1,\dots,d_s)$, a \bemph{Pl\"ucker pair} $(S,T)$ consists of a pair of subsets $S\subseteq T$ such that $S\in {[n] \choose d_i-1}$ and either $T\in {[n] \choose d_{i+1} + 1}$ for $i \in [s-1]$, or  $T\in {[n] \choose d_i + 1}$ for $i \in [s]$.
The \bemph{lambda-values} for the Pl\"ucker pair $(S,T)$ are
\[
\lambda^{S,T}_j : = \fvm(St_j)\odot \fvm(T\backslash t_j).
\]
where $t_1<\dots< t_k$ are the elements of $T\backslash S$.
Formally, $\lambda \in \TT\PP^{|T\backslash S|-1}$ is a vector in tropical projective space. 
\end{definition}

Again we let $\lambda_0 = \lambda_{|T\backslash S|+1}=\infty$.
With this notation we have that the Pl\"ucker-incidence relations (\ref{eq:troplucker}) can be written simply as
\begin{equation}
\bigoplus_i \lambda_i^{S,T}\text{ achieves the minimum at least twice.}
\label{eq:lambda_inci}
\end{equation}

With this language we can formulate our main result for arbitrary rank which is proved using \Cref{thm:hollow} and the fact that one can associate to each Pl\"ucker pair a $(S,T)$ a certain hollow flag matroid encoding the positivity conditions for that pair (see \Cref{def:hp}).

\begin{theorem}
\label{thm:arb}
Let $\fvm\in \FlDr(\fd,n)$ be any valuated flag matroid.
\begin{enumerate}[1.]
	\item If $\fvm$ is totally non-negative then
	 \begin{equation}
\lambda_i^{S,T} \ge \lambda^{S,T}_{i-1} \oplus \lambda^{S,T}_{i+1}
	 \label{eq:tnna}
	 \end{equation}	
 for every Pl\"ucker pair $(S,T)$ and $i\le|T\backslash S|$.
	\item If $\fvm$ induces a Bruhat subdivision then 
	 \begin{equation}
	\lambda_i^{S,T} = \lambda^{S,T}_{i-1} \oplus \lambda^{S,T}_{i+1}
	 \label{eq:brua}
	 \end{equation}
	for every every Pl\"ucker pair $(S,T)$ and $i$ such that $\lambda_i = \lambda_1\oplus\dots\oplus\lambda_j$.
		\item If $\fvm$ induces a flag positroid subdivision then $\fvm\in \FlDr^{\ge0}(\fd,n)$, i.e. 
	 \begin{equation}
	\bigoplus_{i \text{ odd}} \lambda^{S,T}_i   = \bigoplus_{i \text{ even}} \lambda^{S,T}_i 
	 \label{eq:posa}
	 \end{equation}	
\end{enumerate}
\end{theorem} 

From this theorem it follows immediately that \ref{enum:bru} implies \ref{enum:dress}, but we still do no know whether \ref{enum:bru} implies \ref{enum:pl+}.
However, we feel bold to conjecture that the situation in hollow rank extends to arbitrary rank and the converses of \Cref{thm:arb} also hold:

\begin{conjecture}
\label{conj:main}
Let $\fvm\in \FlDr(\fd,n)$ be any valuated flag matroid.
\begin{enumerate}[1.]
	\item $\fvm$ is totally non-negative if and only if	(\ref{eq:tnna}) holds for every Pl\"ucker pair.
	\item $\fvm$ induces a Bruhat subdivision if and only if (\ref{eq:brua}) holds for every Pl\"ucker pair.
	\item The following are equivalent:
	\begin{enumerate}[i.]
		\item $\fvm$ is a valuated flag positroid
		\item $\fvm$ induces a flag positroid subdivison 
		\item $\fvm\in \FlDr^{\ge 0}(\fd;n)$.
	\end{enumerate}
\end{enumerate}
\end{conjecture} 

We want to remark that the first part of \Cref{conj:main} implies the second (see \Cref{coro:conjimp}). 
This follows from the fact that it is enough for the second part to prove it for unvaluated flag matroids (see \Cref{prop:unvaluated}) and that both parts are equivalent for unavluated matroids. 
Similarly, for the equivalence of $ii.$ and $iii.$ in the third part it is enough to prove it for unvaluated matroids, which is precisely \cite[Question 1.6]{BEW}.
This conjecture would characterize all positive notions for partial flag varieties in terms of the quadratic terms appearing in the Pl\"ucker-incidence relations.
The cases where we know \Cref{conj:main} hold are the Grassmannian case \cite{ArkaniHamedLamSpradlin:2021,SpeyerWilliams:2021}, consecutive flags \cite{BEW,JLLO} (equations (\ref{eq:tnna}), (\ref{eq:brua}) and (\ref{eq:posa}) are all equivalent to the three term Pl\"ucker-incidence relations for these cases) and hollow rank (\Cref{thm:hollow}).

Our results only concern type $\textsc{A}$ partial flag varieties, but positivity and Bruhat polytopes have also been studied for generalized flag varieties $G/P$ \cite{chevalier2011total,KodamaWilliams:2015,TsukermanWilliams}.
A more general approach would involve Coxeter flag matroids and their polytopes as described in \cite{BorovikGelfandWhite:2003}. 
The tropicalization of type $\textsc{D}$ Grassmannians has nice interpretation in terms of  $\Delta$-matroids \cite{rincon2012isotropical}, so we believe similar results can be obtained for this case.
This is in stark contrast with the type $\textsc{C}$ Grassmannian, where the polyhedral geometry is poorly understood \cite{balla2023tropical}.

The paper is organized as follows.
Section \ref{sec:background} introduces the necessary background on valuated flag matroids.
Bruhat polytopes are studied in \Cref{sec:bruhat}, which is where \Cref{thm:introbru} is proven.
The matroid subdivisions of $\Delta(1,n-1;n)$ are explained in \Cref{sec:geo}, where we give \Cref{ex:main} showing that notions \ref{enum:tnn} and \ref{enum:bru} differ for non-consecutive rank.
We prove part 2. of \Cref{thm:hollow} in \Cref{sec:hollow_bruh}. 
We introduce valuated flag gammoids in \Cref{sec:gammoid} and use this language to reformulate the parametrization of the tropicalization of Richardson cells in $\TFl^\tnn(1,n-1;n)$ in a way that does not make use of Bruhat intervals. 
Then we finish the proof of \Cref{thm:hollow} in \Cref{sec:hollow+}.
Finally we prove \Cref{thm:arb} in \Cref{sec:arb}, where also give some additional remarks for \Cref{conj:main} in general rank and some final questions.

\subsection*{Acknowledgements}
The author is grateful to Michael Joswig for the support during the initial stages of this project.
The author would also like to thank Lauren Williams for useful discussions about Bruhat polytopes.

\section{Notation}
The following notation will be used throughout:
\begin{itemize}
	\item $[n]=\{1,2,\dots,n\}$
	\item ${[n]\choose d} = \{\text{subsets of } [n] \text{ of size } d\}$
	\item $Si=S\cup\{i\}$ and $T\backslash i = T\backslash \{i\}$.
	\item $\TT = \RR\cup\{\infty\}$ is the tropical semifield where $\oplus =\min$ and $\odot=+$.
	\item $\TP= \RR^n/\RR(1,\dots,1)$ is the tropical projective space (we keep the all $\infty$ vector because the lambda values from \Cref{def:lambda} can be all $\infty$).
	\item Letters in bold correspond to flags of objects, for example $\fvm=(\mu_1,\dots,\mu_s)$ a valuated flag matroid of rank $\fd=(d_1,\dots,d_s)$.
	\item $\fn=(1,\dots,n)$ is the rank of the complete flag.
	\item For a subset $B$ of $[n]$, 
	\[
	e_B:=\sum_{i\in B} e_i
	\]
	is the indicator vector of $B$. 
	\item For a flag of sets $\fB= (B_1,\dots, B_s)$, $e_\fB=e_{B_1}+\dots+e_{B_s}$.
	\item $\Delta(d,n)= [0,1]^n\cap\left\{\sum\limits_{i=1}^n x_i = d\right\}$ is the hypersimplex. 
	\item $\Delta(\fd,n) = \sum\limits_{i=1}^s \Delta(d_i,n)$ is the polytope of the uniform flag matroid of rank $\fd$.
	\item $\Pi_n=\Delta(\fn,n)$ is the permutahedron of dimension $n-1$.
	\item For flag varieties or hypersimplices, we use as arguments either $(\fd,n)$ or equivalently $(d_1,\dots,d_s;n)$.
	\item $\Sym_n$ is the symmetric group.
	\item $\tau_i$ is the transposition $(i,i+1)\in \Sym_n$. 
	\item Other permutations are written in one line notation without parenthesis, for example $\sigma=312$ is the permutation $\sigma \in \Sym_n$ where $\sigma(1)=3$, $\sigma(2)=1$ and $\sigma(3)=2$.
\end{itemize}

\section{Background on valuated flag matroids}
\label{sec:background}

%
%


We assume basic knowledge of tropical geometry, matroid theory and polyhedral subdivisions. For an introduction to these fields and further details we point the reader to \cite{ETC,Tropical+Book} for tropical geometry, \cite{Triangulations} for polyhedral subdivisions and \cite{Oxley:2011,BorovikGelfandWhite:2003} for matroid theory. 

\subsection{Mixed subdivisions}
\label{sec:mixed}
Recall that a function $w:V(P)\to \RR$ on the vertices of a polytope $P\subset \RR^n$ induces a regular subdivision $\subdiv_w(P)$.
The cells in $\subdiv_w(P)$ are the convex hull of all $v\in V$ that minimize $w(v) - x\cdot v$ for some $x\in (\RR^*)^n$, which we write $\subdiv_w(P)^x$ . 
We also get a polyhedral subdivision $\Sigma_w$ of $(\RR^*)^n$, where for every polytope $Q\in \subdiv_w(P)$ there is a cell $Q^\vee\in \Sigma_w$ coonsisting of all $x\in (\RR^*)^n$ such that $\subdiv_w(P)^x=Q$.

Now consider a Minkowski sum of polytopes $P_1+\dots+P_k$. 
Given a tuple of functions $\fw=(w_1,\dots w_k)$ where $w_i:V(P_i)\to \RR$, the (coherent) mixed subdivision $\subdiv_\fw(P_1+\dots+P_k)$ consists of all polytopes of the form 
\[
\subdiv_{w_1}(P_1)^x+\dots+\subdiv_{w_k}(P_k)^x.
\]
The dual complex $\Sigma_\fw$ is the common refinement of all $\Sigma_i$.
Equivalently, mixed subdivisions are the subdivisions induced by the projection of a product of simplices to the Minkowski sum.

\subsection{Flag matroids}
\label{sub:fm}
Matroids have many equivalent definitions.
Of most importance to us, is the polyhedral approach intiated by Gelfand, Goresky, MacPherson and Serganova \cite{GGSM}.
We begin by introducing the polytopes that will play a role. 
First we have the \bemph{hypersimplex} 
\[
\Delta(d,n) = [0,1]^n\cap \{x_1+\dots+ x_n\} = \conv \left\{e_B \mid B \in {[n]\choose d}\right\}
\]
where $e_B$ is the indicator vector of the subset $B$. 
The \bemph{permutahedron} is the convex hull of the orbit of $(1,\dots,n)$ under permuting its coordinates, or, equivalently:
\[
\Pi_n = \Delta(1,n)+\dots+\Delta(n,n).
\]

A \bemph{matroid} $M$ with underlying set $[n]$ is a \bemph{subpolytope} (the convex hull of a subset of vertices) of a hypersimplex $\Delta(d,n)$ such that every edge of $M$ is an edge of the hypersimplex. 
The number $d$ is the \bemph{rank} of a matroid and any $B$ such that $e_B$ is a vertex of $M$ is called a \bemph{basis}, from which you can recover the more traditional notions of matroids.
We write $\BB(M)$ for the set of bases of $M$.

Two matroids $M_1$ and $M_2$ on the same ground set and of ranks $d_1 < d_2$ form a \bemph{quotient} if every flat of $M_1$ is also a flat $M_2$. 
The motivation behind this definition is that if a linear space sits inside another linear space, then the corresponding representable matroids form a quotient.
A \bemph{flag matroid} is $\fM = (M_1,\dots,M_s)$ is a chain of matroids such that any pair of them form a quotient.
The matroids $M_i$ are the \bemph{constituents} of $\fM$.
The \bemph{polytope} $P_\fM$ of a flag matroid $\fM$ is the Minkowski sum $M_1+\dots+M_s$.
For example, if $\fM$ consists of a complete chain of uniform matroids $(U_{1,n},\dots,U_{n,n})$, then $P_\fM$ is the permutahedron $\Pi_n$.
The \bemph{rank} of $\fM$ is the tuple $(d_1,\dots,d_s)$ consisting of the ranks of the constituent matroids.
A polytope is a flag matroid polytope if and only if it is a \bemph{generalized permutahedron} (all edge directions are of the form $e_i-e_j$) and it is the subpolytope of a Minkowski sum of hypersimplices  \cite[Theorem 1.11.1]{BorovikGelfandWhite:2003}.
The \bemph{bases} of a flag matroid $\fM$ are
\[
\BB(\fM) := \bigcup_{i=1}^s \BB(M_i).
\]

A \bemph{flag of subsets} of rank $\fd = (d_1,\dots,d_s)$ is a chain of subsets of $[n]$, $\fB = B_1\subset \dots \subset B_s$ where $|B_i| = d_i$. 
Given a flag of subsets $\fB$, we define $e_\fB := e_{B_1} + \dots+e_{B_s}$. 
The polytope of a flag matroid with uniform constituents $\Delta(\fd, n)$ is the convex hull of $e_\fB$ for all flags of subsets $\fB$.


A \bemph{flag of linear subspaces} of rank $\fd$ over a field $\KK$ is a chain $L_1\subset \dots \subset L_s$ of linear subspaces of a vector space $\KK^n$ such that $\dim(L_i) = d_i$.
A flag matroid $\fM= (M_1,\dots,M_s)$ is \bemph{representable} (or \bemph{realizable}) over afield $\KK$ if there is a flag of subspaces $(L_1, \dots, L_s)\in \Fl(\fd, \KK^n)$ such that the matroid of $M(L_i)$ of $L_i$ is $M_i$.
We write $\fM(\fL)$ for the flag matroid of the flag $\fL$.

When the rank of a flag, whether it is of matroids, sets or linear spaces, is $\fn = (1,\dots,n)$, we call the flag \bemph{complete}.
We call the flag \bemph{consecutive} if the rank is of the form $(d+1,\dots,d+s)$.
On the other extreme, we call the flag \bemph{hollow} if $\fd = (1,n-1)$.

\subsection{Valuated matroids}

A \bemph{valuated matroid} with underlying matroid $M$ is a function $\mu: \BB(M) \to \RR$ such that the regular subdivision it induces on $P_M$ consists exclusively of matroids. 
Equivalently \cite{Speyer:2008}, $\mu$ satisfies the \bemph{Pl\"ucker relations} (\ref{eq:troplucker})
for pairs of subsets $S\in {[n] \choose d-1}$ and $T\in {[n] \choose d+1}$.
An \bemph{initial matroid} of $\mu$ is a matroid appearing in the subdivision, and we write $\mu^x:=\subdiv_\mu(P_M)^x$. 
A valuated matroid can be interpreted as a function from ${[n] \choose d} \to \TT := \RR\cup\{\infty\}$ by setting the value of all non-bases to infinity. 
A matroid can be considered as a valuated matroid with trivial valuation, that is, by setting $0$ to all of its bases and $\infty$ to the non-bases.
\emph{The Dressian} $\Dr(d,n)$ is the space of all valuated matroids of rank $d$ over $[n]$.
It can be given a fan structure as a subfan of the secondary fan of $\Delta(d,n)$, or, equivalently, by the terms in which the three-term Pl\"ucker relations achieve the minimum \cite{OPS}.

A pair of valuated matroids $(\mu_1,\mu_2)$ of ranks $d_1 < d_2$ form a \bemph{quotient} if they satisfy the \bemph{tropical incidence relations} (\ref{eq:troplucker}),
for every pair of subsets $S\in {[n] \choose d_1-1}$ and $T\in {[n] \choose d_2+1}$.
If we restrict to trivially valuated matroids we recover the concept of matroid quotient described above.

A \bemph{valuated flag matroid} $\fvm = (\mu_1,\dots,\mu_s)$ is a chain of valuated matroids such that every pair of them form a quotient.
\emph{The flag Dressian} $\Fl\Dr(\fd;n)$ consists of all valuated flag matroids of rank $\fd$.
A sequence of valuated matroids $\fvm = (\mu_1,\dots,\mu_s)$ is a flag valuated matroid if and only if the supports $(M_1,\dots, M_s)$ form a flag matroid polytope and the mixed subdivision $\subdiv_\fvm(M_1+\dots+M_s)$ consists exclusively of flag matroid polytopes \cite[Theorem 4.4.2]{BEZ}.
Therefore a flag matroid can also be thought of as a function from the vertices of a flag matroid polytope to $\RR$.
Again we write $\fvm^x:=\subdiv_\fvm(P_\fM)^x$ for \bemph{initial flag matroids}. 

The \bemph{tropical linear space} $L_\mu$ given by a valuated matroid $\mu$ is the set of points $x\in \TT^n$ that satisfy that
\[
\bigoplus_{i\in [n]}x_i \odot \mu_{T\backslash i} \enspace \text{ achieves the minimum at least twice.}
\]
It is the subcomplex of $\Sigma_\mu$ (as defined in \Cref{sec:mixed}), consisting of all cells $P_M^\vee$ where $M$ is a loop free matroid in the subdivision of induced by $\mu$ \cite{Speyer:2008}. 
 
Since all tropical linear spaces contain $\RR(1,\dots,1)$ in the lineality space, we think of tropical linear spaces as sitting inside the \bemph{tropical projective plane} 
\[\TP := \TT^n/ \RR(1,\dots,1).\] 
Similar to the non-tropical case, a tropical linear space is contained in another tropical linear space if and only if the corresponding valuated matroids form a quotient \cite{Haque}.
Hence valuated flag matroids $(\mu_1,\dots,\mu_s)$ are cryptomorphic to flags of tropical linear spaces $L_{\mu_1}\subseteq \dots L_{\mu_s}$.

A valuated matroid $\mu$ is said to be \bemph{representable} (or \bemph{realizable}) by a linear subspace $L$, if $L_\mu = \trop(L)$.
A beautiful fact about the tropicalization of linear spaces is that tropicalizing commutes with taking Pl\"ucker coordinates, i.e. the tropicalization of a linear space is given by the valuated matroid obtained by taking the valuation of the Pl\"ucker coordinates the linear space \cite{SS}.
%


%
%
\subsection{Positivity}

The following two notions of positivity exist for the flag variety.
The \emph{Pl\"ucker non-negative flag variety} $\Fl^{\ge0}(\fd,\RR^n)$ consists of all flag of linear spaces whose Pl\"ucker cooridnates are non-negative.
In other words, $\Fl^{\ge0}(\fd,\RR^n)$ consists of flags of spaces that come from the row span of a matrix in $\RR^{n\times n}$ such that all top $d_i\times d_i$ minors are non-negative for every $i\in[s]$.
A stronger notion, is the \emph{totally non-negative} (or just $\tnn$) \emph{flag variety} , $\Fl^{\tnn}(\fd, \RR^n)$, which consists of all flags of linear spaces which can be extended to a complete flag in $\Fl^{\ge 0}(\fn,\RR^n)$, i.e. flags that can be realized as the row span of a totally non-negative matrix.
These two notions agree if and only if the rank $\fd$ consists of consecutive integers \cite{bloch2022two}.

A flag matroid $\fM$ such that $\fM= \fM(\fL)$ for a flag $\fL\in \Fl^{\ge0}(\fd,\RR^n)$ is called a \emph{flag positroid}.
For the Grassmannian case $\fd= (d)$, this recovers the notion of \emph{positroid} introduced by Postnikov \cite{postnikov}.
If $\fL\in \Fl^{\tnn}(\fd,\RR^n)$, then $\fM$ belongs to a smaller class of matroids (when $\fd$ is not consecutive) which corresponds to flag matroids whose polytope is a Bruhat polytope (see \Cref{prop:richard}).


The study of positivity for valuated matroids started with \cite{SpeyerWilliams:2005}, where the following framework for the positive part of any tropical variety.
Let $\C = \CC\ldb t\rdb$ be the field of generalized Puiseux series with complex coefficients\footnote{The standard field with non-trivial valuation which is mostly used is in tropical geometry is the field of Puiseux series.
However, this field has the disadvantage that the valuation is not surjective; the value group is only $\QQ$ and not all of $\RR$.
Therefore, only valuated matroids with rational are representable over this field. 
We can avoid dealing with this inconvenience by working with a larger field whose valuation is surjective.}
(see \cite{Markwig:2010} for details on this field)
and let $\R_{\ge 0}$ subset of $\C$ consisting of all series whose leading coefficient is a positive real number (and $0$). 
Given a subvarierty $X=V(I)$ of $\C^n$, the non-negative part $\trop^+(X)$ of the tropicalization of $\trop(X)$ has the following two equivalent interpretations \cite{SpeyerWilliams:2005}:
\begin{itemize}
	\item The image under the valuation map of $X\cap \R_{\ge 0}^n$.
	\item The intersection
	\begin{equation}
		\bigcap_{f\in I} \trop^+(f)
	\label{eq:trop+}
	\end{equation}

	where $\trop^+(f)$ is the set of $w\in \RR^n$ such that $\init_w(f)$ does not contain any non-zero polynomial in $\R_{\ge 0}$. 
\end{itemize}
This idea can be extended to projective varieties and their tropicalizations landing in tropical projective space.
Just like with usual tropicalization, the intersection in (\ref{eq:trop+}) may require infinite elements $f\in I$.
A subset of $I$ such that the intersection of $\trop^+(f)$ for every element of that subset produces $\trop^+(X)$ is called a \emph{positive-tropical basis} (see \cite{brandenburg2022tropical} for a comparison with usual tropical bases).

When applied to the flag variety, we obtain the \emph{non-negative tropical flag variety} $\TFl^{\ge 0}(\fd,n) :=\trop^+(\Fl(\fd,\C_{\ge^0}^n))$.
We call a valuated flag matroid $\fvm \in \TFl^{\ge 0}(\fd,n)$ a \emph{valuated flag positroid}.
By the observations above, $\TFl^{\ge 0}(\fd,n)$ is the image under the valuation map of $\Fl^{\ge0}(\fd,\RR^n)$.
The \emph{totally non-negative tropical flag variety} $\TFl^\tnn(\fd,n)$ is the image under the valuation map of $\Fl^{\tnn}(\fd,\RR^n)$.
A valuated flag matroid $\fvm$ is called $\tnn$ if it is in $\TFl^\tnn(\fd,n)$.

The \emph{non-negative flag Dressian } is 
\[
\FlDr^{\ge 0}(\fd,n) :=  \cap_f \trop^+(f)
\]
where the intersection ranges through all Pl\"ucker and incidence relations $f$. 
This corresponds to flag matroids over the hyperfield of signed tropical numbers with a positive orientation \cite{JO}.

It was simultaneously proven in \cite{ArkaniHamedLamSpradlin:2021} and \cite{SpeyerWilliams:2021} that the 3-term Pl\"ucker relations form a positive-tropical basis for the Grassmannian.
Similarly, it was simultaneously proven in \cite{JLLO} and \cite{boretsky} that the 3-term incidence relations form a positive-tropical basis for the complete flag variety.
This was generalized to flag varieties of consecutive rank \cite{BEW}.
The positive tropicalization of the three term incidence relations,
\begin{equation}
\mu(Sj)\odot\mu(Sik) =\mu(Si)\odot\mu(Sjk)\oplus\mu(Sl)\odot\mu(Sij)
\label{eq:3tin}
\end{equation}
for $i<j<l$, are equivalent to Equations (\ref{eq:tnna}), (\ref{eq:brua}) and (\ref{eq:posa}). 
Whether the incidence relations together with the Pl\"ucker relations form a tropical basis of the flag variety for any rank $\fd$ is essentially the content of part 3. of \Cref{conj:main}.

We say that a subdivision of flag matroid polytope $\fM$ is a \emph{flag positroid subdivision} if all the polytopes are flag positroids and it is a \emph{Bruhat subdivision} if all the polytopes are Bruhat polytopes.
For consecutive rank, the non-negative flag Dressian is equal to the non-negative tropical flag variety and it corresponds exactly to valuated flag matroids inducing a flag positroid subdivision, or equivalent a Bruhat subdivision \cite{ArkaniHamedLamSpradlin:2021,BEW,JLLO,SpeyerWilliams:2021}.
As Bruhat polytopes differ from flag positroid polytopes for non-consecutive rank, we have that valuated flag matroids inducing Bruhat subdivisions form a strictly smaller class than subdivisions inducing flag positroid subdivisions.

\begin{remark}
In \cite{tamayo} the term flag positroid is used for positroids which form a matroid quotient.
However one can easily construct positroids which form a quotient but do not consitute a flag positroid in our terms, for example the flag matroid with bases ${1,2,3,12,13}$ (see \cite[Example 3.3]{JLLO}. 
Therefore we restrict the term flag positroid to this stronger realization condition following \cite{BEW}.
\end{remark}

\subsection{Matroid constructions}
Given a valuated matroid $\mu$ of rank $d$, its dual $\mu^*$ is a valuated matroid of rank $n-d$ such that 
\[
\mu^*(B) := \mu([n]\backslash B).
\]
Given an element $i\in [n]$, which is not a coloop, the \bemph{deletion} of $i$ from $\mu$ is the valuated matroid $\mu\backslash i$ over $n\backslash i$ such that 
\[
\mu\backslash i(B) := \mu(B)
\]
where $i\notin B$. If $i$ is not loop, the \bemph{contraction} of $i$ from $\mu$ is the valuated matroid $\mu/i$ over $n\backslash i$ such that 
\[
\mu/ i(B) := \mu(B\backslash i)
\]
where $i\in B$. 
If $i$ is either a loop or a coloop,  $\mu\backslash i = \mu/ i(B)$. 
Given a subset $I\subseteq [n]$, the valuated matroid $\mu\backslash I$ is the successive deletion of all elements in $I$, and similarly $\mu/I$ is the succesive contraction by all elements in $I$.
Any valuated matroid of the form $\mu/ I\backslash I$ is called a \bemph{minor}.
Sometimes is convenient for notation purposes to write $\mu|_I := \fvm\backslash ([n]\backslash I)$ the \bemph{restriction} of $\fvm$ to $I$.

The \bemph{translation} of tropical linear space $L_\mu+x$ by a vector $x\in \RR^n$ is also tropical linear space.
It corresponds to the valuated matroid $\mu+x$ where
\[
(\mu+x)(B) := \mu(B)+e_B\cdot x.
\]

All of the above constructions can be extended to valuated flag matroids by applying it to each constituent. 
That is
\begin{itemize}
	\item The \bemph{dual} is $\fvm^*:=(\mu_s^*,\dots,\mu_1^*)$ with rank $\fd^*=(n-d_s,\dots,n-d_1)$.
	\item The \bemph{minors} $\fvm/I\backslash J = (\mu_1/I\backslash J,\dots,\mu_s/I\backslash J)$, after possibly removing repeated elements.
	\item The \bemph{translation} $\fvm+x=(\mu_1+x,\dots,\mu_s+x)$. 
\end{itemize}
The details that these are indeed flag valuated matroids can be found in \cite{BEZ} or \cite{CO}.
We will make use of the following theorem:

\begin{theorem}{\cite[Theorem 5.1.2]{BEZ}}
\label{prop:bez}
Two matroids $\mu_1$ and $\mu_2$ of rank $d$ and $d+1$ form a valuated flag matroid $(\mu_1,\mu_2)$ if an only if there exists a valuated matroid $\nu$ over $[n+1]$ such that $\mu_1=\nu/(n+1)$ and $\mu_2=\nu\backslash(n+1)$ 
\end{theorem}

For realizable matroids, these construction have all an interpretation at the level of linear spaces. If $\mu=\mu(L)$, then 
\begin{itemize}
	\item $\mu^*$ is the the valuated matroid of the orthogonal complement of $L$.
	\item $\mu\backslash i$ is the valuated matroid of the projection of $L$ into $\RR^{[n]\backslash i}$,
	\item  $\mu/i$ is the valuated matroid of $L\cap\{x_i =0\}$
	\item $\mu+x$ is the valuated matroid of the result of scaling $L$ in the coordinate $i$ by $t^{x_i}$ for every $i$.
\end{itemize}

A crucial fact is that all of these notions preserve positivity \cite[Proposition 3.5]{ardila2016positroids}.
So this implies that if $\fvm$ is $\tnn$, minors, duals or translations of $\fvm$ are also $\tnn$ and if $\fvm$ is a flag positroid, minors, duals and translations of $\fvm$ are also flag positroids.

Finally, we want to remark that the homogenous nature of the lambda-values of $\fvm$ as defined in \Cref{def:lambda} implies that they are invariant under translation.
This is a fact that we will often make use of ahead.

\section{Bruhat polytopes}
\label{sec:bruhat}
The following is an ad hoc introduction to Bruhat (interval) polytopes for type $\textsc{A}$, in a purely combinatorial fashion. 
For simplicity, we drop the ``interval'' and call them just Bruhat polytopes, following \cite{JLLO}.
For details on generalized Bruhat polytopes for a general flag variety $G/P$ we suggest \cite[Section 6]{TsukermanWilliams}.


A \bemph{word} for a permutation $v$ is sequence of $w=(w_1,\dots,w_m)$ where $v=w_1\cdots w_m$ using the precomposition product and each $w_i$ is a transposition of the form $\tau_j := (j,j+1)$. 
We say the word is \bemph{reduced} if it is length $m$ is minimal among all words for $v$.
A subword of $w$ is a (not necessarily consecutive) subsequence $w'=(w_{i_1},\dots w_{i_{m'}})$ where $1\le i_1 < \dots < i_{m'} \le m$.
The \bemph{Bruhat order} on the symmetric group $\Sym(n)$ is given by $u\le v$ if and only if for any (equivalently every) reduced word $w$ for $v$ there is a subword $w'$ for $u$. 
The \bemph{length} of a permutation $v$ is the minimal length of a word for $v$.

Two different definitions have been used simultaneously in the literature to assign a polytope to a Bruhat interval.
These are based on two different ways of identifying $\Sym(n)$ with the vertices of the permutahedron $\Pi_n$.
\begin{align*}
\phi:\Sym_n &\to \vertices(\Pi_n)\\
x &\mapsto (x(1),\dots,x(n))
\end{align*}
and
\begin{align*}
\psi:\Sym_n &\to \vertices(\Pi_n)\\
x &\mapsto x\cdot e_\fn
\end{align*}
where $e_{[\fn]}:=e_{[1]}+\dots+e_{[n]}=(n,\dots,1)$ and $x\cdot e_\fn$ is given by the action of $\Sym_n$ on $\RR^n$ of permuting coordinates. 
Following terminology introduced in \cite{BEW}, given two permutations $u\le v$, we call
\[
\BP_{u,v} := \conv(\phi([u,v])) = \conv\{(x(1),\dots,x(n)) \mid x\in [u,v]\}
\]
the \bemph{Bruhat polytope} and
\[
\tBP_{u,v} := \conv(\psi([u,v])) = \conv\{x\cdot e_\fn \mid x\in [u,v]\}
\]
the \bemph{twisted Bruhat polytope}.

For general flag varieties $G/P$, the twisted Bruhat polytope has been generalized, sometimes called \bemph{generalized Bruhat interval polytopes}.
This works as follows for type $\textsc{A}$.
Let $\fd=(d_1,\dots,d_s)$ be a rank vector and let $e_\fd:=e_{[d_1]}+\dots+e_{[d_s]}$. 
Consider the function
\begin{align*}
\pi^\fd: \vertices(\Pi_n) &\to \vertices(\Delta(\fd,n))\\
p &\mapsto (f^\fd(p(1)),\dots, f^\fd(p(1))).
\end{align*}
where $f: [n]\to [s]$ is the function where $\hat{\phi}^\fd(i)$ is the $(n+1-i)$-coordinate of $e_\fd$, i.e. the number of times $n+1-i$ appears in the flag $[d_1]\subset \dots \subset [d_s]$.
At the level of flag matroid polytopes, $\pi^\fd$ is essentially a forgetful morphism, taking $P_\fM$ to $P_{\fM'}$ where $\fM'$ is the result of dropping from $\fM$ all constituents whose rank is not in $\fd$.
The stabilizer of $\Sym_n$ with respect to $e_\fd$ subudivides $\Sym_n$ into cosets which are the fibers of the map $\pi^\fd\circ\psi$. 
We call $v\in \Sym_n$ a \bemph{minimal coset representative} of rank $\fd$ if it is the smallest element under the Bruhat order of its coset, i.e. its fiber under $\pi^\fd\circ\psi$.
Let $\tBI(\fd,n)$ be the set of Bruhat intervals $[u,v]$ where $v$ is a minimal coset representative of rank $\fd$. 

\begin{definition}
\label{def:tbru}
Given an interval $[u,v]\in\tBI(\fd,n)$, the (generalized) twisted Bruhat polytope $\tBP_{u,v}^\fd$ is
\[
\tBP_{u,v}^\fd := \conv(\pi^\fd\circ\psi([u,v])) = \conv\{x\cdot e_\fd \mid x\in [u,v]\}.
\]
\end{definition}

Because $\phi$ is the most natural way of identifying permutations with the vertices of the permutahedron, the untwisted variant of Bruhat polytopes is usually used as definition in the complete flag case, \cite[Definition A.5]{KodamaWilliams:2015} and \cite[Definition 2.2]{KodamaWilliams:2015}.
However, for the general case, the twisted variant is used because of the connection to Richardson varieties, \cite[Definition 6.9]{KodamaWilliams:2015} and \cite[Definition 7.8]{KodamaWilliams:2015}.
These are cells in a partition of $\Fl^\tnn(\RR^n,\fd)$ which are indexed by a pair $u\le v$ with $v$ a minimal coset representative. 
Concretely, we have:

\begin{theorem}{\cite[Corollary 6.11]{KodamaWilliams:2015}}
\label{prop:richard}
Let $\fL$ be a flag of linear spaces in the Richardson variety $\mathcal{R}_{u,v} \subseteq \Fl^{\tnn}(\RR^n,\fd)$.
If $\fM$ is the flag matroid of $\fL$, then $P_\fM = \tBP_{u,v}$.
\end{theorem}

We have in particular that twisted Bruhat polyopes are flag matroid polytopes and when $\fd$ is consecutive, twisted Bruhat polytopes are exactly the same as flag positroid polytopes.
Another important result about twisted Bruhat polytopes which we will use later is the following:

\begin{theorem}{\cite[Theorem 7.13]{TsukermanWilliams}}
\label{prop:faces}
The face of twisted Bruhat polytope is a twisted Bruhat polyope.
\end{theorem}

Even though $\BP_{u,v}$ may differ from $\tBP_{u,v}$, having these two different definitions is not a big issue for the complete flag case because the set of Bruhat polytopes is the same as the set of twisted Bruhat polytopes. 
To see this,notice that $\phi^{-1}\circ\psi$ is the function $\iota:\Sym_n\to \Sym_N$ given by
\begin{equation}
\iota: x\mapsto (i\mapsto n+1-x^{-1}(i)).
\label{eq:iota}
\end{equation}

This function consists of the composition of the order preserving involution $x\mapsto x^{-1}$ and the order reversing involution $x\mapsto (i\mapsto n+1-x(i))$.
Therefor $\iota$ reverses the Bruhat order.
As cosets are of size 1 in the complete flag case, we do not impose any conditions on $v$ to form a twisted or untwisted Bruhat polytope. 
So we have $\tBP_{u,v} = \BP_{\iota(u),\iota(v)}$ and hence the set of Bruhat polytopes is the same as the set of untwisted Bruhat polytopes for this case.

For non-complete flags, we can also define untwisted Bruhat polytopes using $\phi$. 
Namely, let $\BI(\fd,n)$ be the set of Bruhat intervals $[u,v]$ where $v$ is the minimal element in its fiber of $\pi^\fd\circ\phi$.
\begin{definition}
\label{def:bru}
Given an interval $[u,v]\in \BI(\fd,n)$, the (generalized) \bemph{Bruhat polytope} $\BP_{u,v}^\fd$ is
\[
\BP^\fd_{u,v} := \conv(\pi^\fd\circ\phi([u,v]))
\]
\end{definition}

%
%
%
%
%
%
%
%
%
%
%
%
%

\begin{example}
\label{ex:bru1}
Consider the Bruhat interval $[2134,3241]$, using one line notation for permutations. 
This interval consists of eight permutations
\[
[2134,3241] = \{2134,2314,2143,3124,2341,3142,3214,3241\}.
\]
For rank $\fd=(1,3)$, we have that the corresponding untwisted Bruhat polytope is
\[
\BP^{(1,3)}_{2134,3241} := \conv((1,0,1,2),(1,0,2,1),(1,1,0,2),(1,1,2,0))
\]
and the twisted variant is
\[
\tBP^{(1,3)}_{2134,3241} := \conv((1,2,1,0),(1,2,0,1),(1,1,2,0),(0,2,1,1),(1,0,2,1),(0,1,2,1))
\]

\begin{figure}[h]
	\centering
		\includegraphics[width=0.95\textwidth]{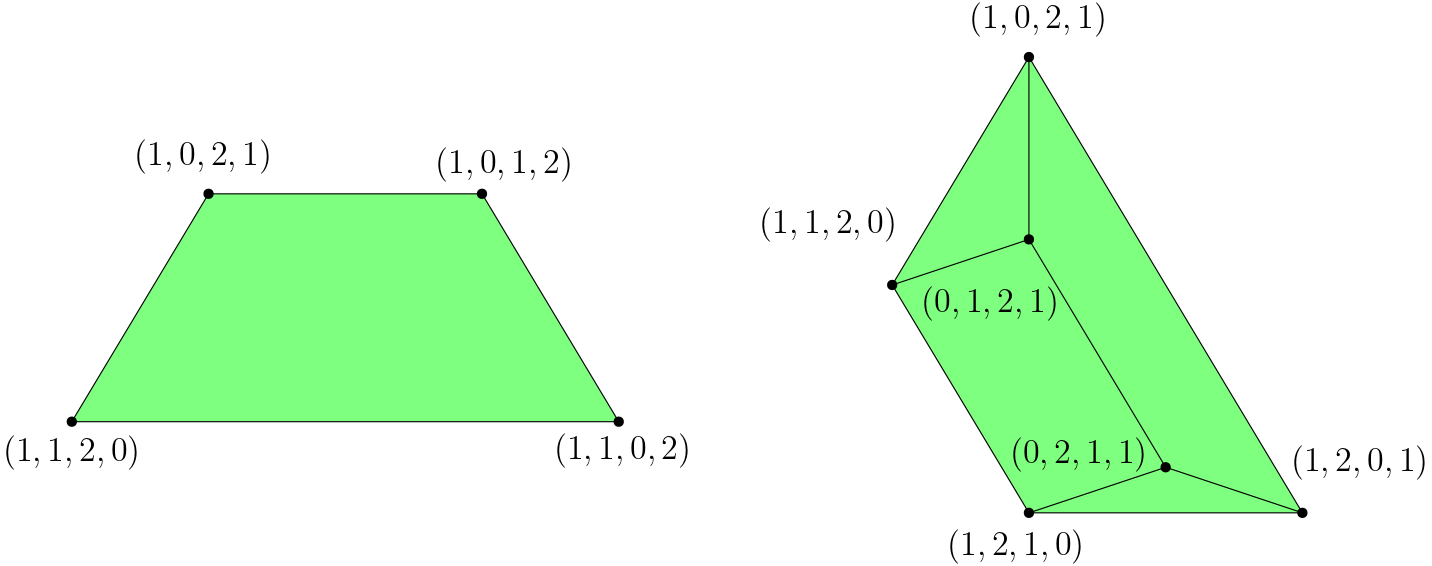}
	\caption{The Bruhart polytopes $\BP^{(1,3)}_{2134,3241}$ (left) and $\tBP^{(1,3)}_{2134,3241}$ (right).}
	\label{fig:Brupolys}
\end{figure}

\end{example}

The same argument used in the complete flag case does not work to prove that twisted and untwisted Bruhat polytopes give the same set for general $\fd$, because given $u\le v$ where $v$ is a minimal coset representative,
then $\iota(v)\le\iota(u)$ does not necessarily satisfy that $\iota(u)$ is the minimal element in its fiber under $\pi^\fd\circ\phi$.
However, we do have the following:

\begin{theorem}
\label{thm:twisted}
The set of Bruhat polytopes and the set of twisted Bruhat polytopes are the same for arbitrary rank $\fd$.
\end{theorem}

\begin{proof}

Let $[u,v]\in \tBI_{\fd,n}$ and consider the twisted Bruhat polytope $\tBP^\fd_{u,v}$.
We will construct another Bruhat interval $[u',v']\in \BI_{\fd,n}$ such that $\tBP^\fd_{u,v} =\BP^\fd_{u',v'}$.

Let $\fM$ be the flag matroid such that $P_\fM = \tBP^\fd_{u,v}$.
By \Cref{prop:richard}, $\fM$ is $\tnn$.
Therefor the dual $\fM^*$ is $\tnn$, and hence there is a Bruhat interval $[a,b] \in \tBI_{\fd^*,n}$ such that $P_{\fM^*}=\tBP^{\fd^*}_{a,b}$.

The key observation is that from (\ref{eq:iota}) we obtain the following commutative diagram:
\begin{equation}
\begin{tikzcd}
\Sym_n \arrow{r}{\psi} \arrow{d}{\_^{-1}} & \vertices(\Pi_n) \arrow{r}{\pi^{\fd^*}} \arrow{d}{\_^{*}} & \vertices(\Delta(\fd^*,n))  \arrow{d}{\_^{*}} \\%
\Sym_n \arrow{r}{\phi} & \vertices(\Pi_n) \arrow{r}{\pi^{\fd}} & \vertices(\Delta(\fd,n))
\end{tikzcd}
\label{eq:comm}
\end{equation} 

Because $b$ is the minimal element in its fiber of $\pi^{\fd^*}\circ\psi$ and $\_^{*}$ is bijective, we have that $b^{-1}$ is minimal in the fiber of $\pi^{\fd}\circ\phi$.
Hence $[a^{-1}, b^{-1}]\in \BI(\fd,n)$. 
Therefor we have that 
\begin{align*}
P_{\fM} &=  (\tBP^{\fd^*}_{a,b})^* \\
&= \conv(\pi^{\fd^*}\circ\psi([a,b]))^* \\
&= \conv(\pi^\fd\circ\phi([a^{-1},b^{-1}]))\\
&= \BP^\fd_{a^{-1},b^{-1}}
\end{align*}
is an untwisted Bruhat polytope, as we wanted.

We have now proven that every twisted Bruhat polytope is also an untwisted Bruhat polytope, but the other direction follows using the same construction since the vertical arrows in (\ref{eq:comm}) are bijections. 
\end{proof}

\begin{example}
Recall the Bruhat polytopes computed in \Cref{ex:bru1} for the interval $[2134,3241]$.
The matroid $\fM$ whose polyope $P_\fM= \BP^{(1,3)}_{2134,3241}$ has bases 
\[
\BB(\fM) = \{3, 4, 123, 124, 134\}.
\]
Its dual $\fM^*$ has bases 
\[
\BB(\fM^*) = \{2, 3, 4, 123, 124\}.
\]
A simple verification yields that for the Bruhat polytope for the interval $[2314,2431]$ satisfies $\BP^{(1,3)}_{[2314,2431} = P_{\fM^*}$.
The inverses of $[2314,2431]$ are 
\[
[3214, 4132]=\{3214, 3142, 4123, 4132\}.
\]
and indeed we have that 
\[
\BP^{(1,3)}_{2134,3241}=\tBP^{(1,3)}_{3214, 4132}.\]

Similarly the flag matroid $\fM'$ that satisfies $P_{\fM'}= \tBP^{(1,3)}_{2134,3241}$ has bases 
\[
\BB(\fM') =\{2, 3, 123, 124, 134, 234\}.
\]
The polytope of its dual is the twisted Bruhat polytope of the interval $[1243,4132]$ and by taking inverses we obtain 
\[\tBP^{(1,3)}_{2134,3241}=\BP^{(1,3)}_{1243, 2431}.
\]
\end{example}

\section{The geometry of hollow flag matroids}
\label{sec:geo}
In this section we study hollow valuated flag matroids $(\mu,\nu)\in \Dr(1,n-1;n)$.
These correspond to a point inside the tropical hyperplane in $\TP$.
The vertices of $\Delta(1,n-1;n)$ are all points with one coordinate equal to $0$, one coordinate equal to $2$ and the other $(n-2)$ coordinates equal to $1$.
However, to make notation simpler and more intuitive, we are going to work with a translation of $\Delta(1,n-1;n)$ (and all matroid polytopes of rank $(1,n-1)$) by $(-1,\dots,-1)$ so that points have only two non-zero coordinates, which are equal to $-1$ and $1$. 
We make a slight abuse of notation by also calling this polytope $\Delta(1,n-1;n)$.
This way $\Delta(1,n-1;n)$ is a centrally symmetric polytope and its matroid subdivisions are induced by the coordinate hyperplanes as we explain below. 
In \Cref{fig:subdiv} we see a finest subdivision of the cuboctahedron $\Delta(1,3;4)$ where two of the coordinate hyperplanes each simultaneously slice up $\Delta(1,3;4)$ in half across a hexagon.
Two of the polytopes are the Minkowski sum of a tetrahedron with a segment while the other two are the Minkowski sum of two triangles.

\begin{figure}[h]

	\begin{center}
		\includegraphics[width=0.45\textwidth]{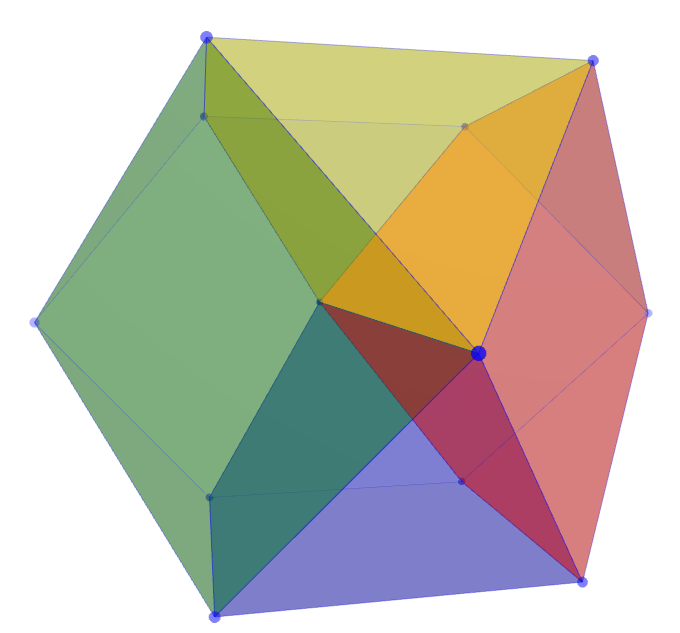}
		\caption{A finest matroid subdivision of the cuboctahedron $\Delta(1,3;4)$.}
		\label{fig:subdiv}
	\end{center}
	
\end{figure}


Every subpolytope of the simplices $\Delta(1,n)$ and $\Delta(n-1,n)$ is a matroid.
By equation (\ref{eq:troplucker}), flag matroids $\fM = (M,N)$ of rank $(1,n-1)$ are indexed by pairs of subsets $(I,J)$ of $[n]$ such that $|I\cap J|\ne 1$, where $I$ is the set of non-loops of $M$ and $J$ is the set of non-coloops of $N$. 
The polytope $P_\fM$ is full dimensional if and only if $I\cup J = [n]$ and $I\cap J \ne \emptyset$.


The most simple type of non-trivial polytopal subdivision is splitting a polytope in two by a hyperplane. 
The matroid subdivisions that arise in $\Delta(1,n-1;n)$ are relatively simple, in that they all consists of the a matroid polytope being simultaneously subdivided by $k$ hyperplanes into $2^k$ pieces for some $k\le n-2$. 

\begin{proposition}
\label{prop:holgeo}
Let $\fM = (M,N)$ be a hollow flag matroid. Then all of its matroid subdivisions are induced by a hyperplane arrangement with hyperplanes of the form $\{x_i = 0\}$.   
\end{proposition}
\begin{proof}
Consider a valuated flag matroid $\fvm = (\mu,\nu)\in \Dr(1,n-1;n)$ over $\fM$.
Lets begin with the full support case, i.e. when $\fM$ is the uniform matroid $U_{1,n-1;n}$.

In this case the tropical linear space $L_\nu$ is a standard tropical hyperplane in $\TP$ centered at $-(\nu([n]\backslash 1),\dots, \nu([n-1]))$ and $L_\mu$ is just the point $(\mu(1),\dots, \mu(n))$.
The complex $\Sigma_{\mu}$ is the inner normal fan of $\Delta(1,n)$ centered at the point $L_\mu$ and $\Sigma_{\nu}$ is the inner normal fan of $\Delta(n-1,n)$ centered at the apex of $L_\nu$ (after taking the quotient by the lineality space $\RR(1,\dots,1)$).
In particular, the rays of $\Sigma_{\nu}$ are in direction $e_i$ while the rays of $\Sigma_{\mu}$ are in direction $-e_i$.
The facets of the mixed subdivision $S_\fvm$ correspond to vertices of $\Sigma_\fvm$, that is, intersections of cones $C_1 \in \Sigma_{\mu}$ and $C_2 \in \Sigma_{\nu}$ which consist of a single point.
So the subdivision $S_\fvm$ only depends on the cell of $L_\nu$ where $L_\mu$ is located.

The set $I\subset[n]$ of numbers $i$ for which $\lambda_i$ does not attain the minimum in (\ref{eq:lambda_inci}) is the set of rays $e_i$ spawning the cone of $L_{\nu}$ whose interior contains the point $L_{\mu}$.
For every subset of $J\subset I$ there is one vertex in $\Sigma_\fvm$, namely, $p^J$ where $p^J_i = \mu(i)$ for $i\in J$ and $p_i = -\nu([n]\backslash i)$ for $i\notin J$. 
In particular we have $2^{|I|}$ facets. 
The point $p^J$ corresponds to the flag matroid $\fM_J$ with bases
\[
\BB(\fM_J) = \{\{i\}\mid i\in (J\cup ([n]\backslash I))\} \cup \{[n]\backslash i\mid i\notin J \}.
\] 
So the corresponding matroid polytope satisfies the inequalities $x_i \le 0$ for $i\in I\backslash J$ and $x_i\ge 0$ for $i\in J$.
Taking all the $2^{|I|}$ facets together we see that the subdivision is the result of simultaneously cutting $\Delta(1,n-1;n)$ by the hyperplanes given by $\{x_i=0\}$ for $i\in I$.

If $\fvm$ does not have full support, the subdivision is still given by the simultaneously splitting with the hyperplanes $\{x_i=0\}$ for all $i$ where $\lambda_i$ does not attain the minimum but is finite. For infinite $\lambda_i$ we have that all of $P_\fM$ is on the halfplane $x_i\ge 0$ if $\nu_i =\infty$ and on the halfplane $x_i\le 0$ if $\mu_i=\infty$ (in particular it is inside the hyperplane if both Pl\"ucker coordinates are infinite).   
\end{proof}

The following example shows total non-negativity, is not determined by the subdivision of the matroid polytope.`
\begin{example}
\label{ex:main}
Consider $\nu = (0,0,0,0) \in \Dr(3,4)$,  $\mu = (2,1,0,0)\in \Dr(1,4)$ and $\mu'=(1,2,0,0)\in \Dr(1,4)$. 
We can check that both $(\mu,\nu)$ and $(\mu',\nu)$ satisfy (\ref{eq:troplucker}), so they are both in the flag Dressian $\FlDr(1,3;4)$.
Moreover, as points, both $\mu$ and $\nu$ lie in the same cone of $L_\nu$, so the corresponding subdivision of $\Delta_{1,3;4}$ is the same.
However, by \Cref{thm:hollow}, $(\mu,\nu)$ can be completed to a complete flag $(\mu,\theta,\nu) \in \TFl^{\ge0}(1,3;4)$. 
For example, we can take 
\[
\theta(ij) =\begin{cases}
			1, & \text{if }  ij=34\\
            0, & \text{otherwise.}
		 \end{cases}
\]
But $(\mu',\nu)$ does not satisfy (\ref{eq:tnnh}), as 
\[\lambda_1 =1 < \lambda_0\oplus\lambda_2 =\infty\oplus 2=2.\] 
So $(\mu',\nu)$ is not in $\TFl^{\ge 0}((1,3),4)$.
The key difference here is that there are different possible subdivisions of $\Pi_n$ that can arise from completing $(\mu,\nu)$ and $(\mu',\nu)$, despite them having the same subdivision on $\Delta(1,3;4)$.
This shows that the polyhedral geometry does not determine whether a valuated flag matroid is a valuated positroid for non-consecutive rank.

Let us see in detail what is going on with this example.
The tropical linear space $L_\nu$ consists of the 6 cones of dimension 2 spanned by 2 of the 4 coordinate axes in $\TT\PP^3$.
Let us examine what are the possible rank-2 tropical linear spaces $L_\theta$ that are inside $L_\nu$ and contain a point $(x,y,0,0)$.
\Cref{fig:13} shows tropical linear spaces of rank 2 which complete $(\mu,\nu)$ and $(\mu', \nu)$.

The tropical linear space $L_\theta$ is determined by its bounded unique bounded cell. 
Lets call $S$ that bounded cell, which is either a segment in direction $e_1+e_2$ or a point.
If $S$ is a point it must be either $(0,y)$ or $(x,0)$.
If $S$ is a segment, it can be a secant of the rectangle $R$ inside $C$ which has opposite corners $(0,0)$ and $(x,y)$ (ignoring the last two coordinates which are constant $0$ in $C$).
Otherwise, it contains $(x,y)$ or $(0,0)$ in its interior (in the latter case the bounded cell is also inside the cone spanned by $e_3$ and $e_4$).

\begin{figure}[h]
	\centering
		\includegraphics[width=0.9\textwidth]{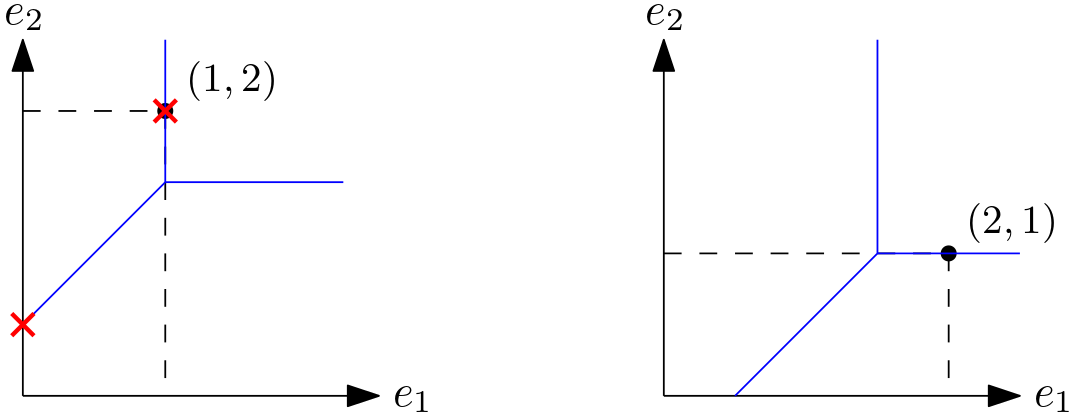}
	\caption{Left: two points (red crosses) corresponding to non-Bruhat polytopes when completing $(\mu',\nu)$. Right: A flag completing $(\mu,\nu)$ producing a Bruhat subdivision.}
	\label{fig:13}
\end{figure}

If one of the extremes $p$ of $S$ lies in the interior of the side of $R$ between $(0,0)$ and $(0,y)$, then the cell in the subdivision of $\Pi_4$ induced by $(\mu,\theta,\nu)$ given by $p$ is the flag matroid with bases $\{3,4,13,14,23,24,34,123,124,234\}$. 
This is not a Bruhat polytope, because it contains the points $(1,2,3,4)$ and $(3,2,1,4)$ but not $(2,1,3,4)$ and $(3,1,2,4)$.

If the other extreme $q$ of $S$ lies in the interior of the side of $R$ between $(x,0)$ and $(x,y)$, then the cell in the subdivision of $\Pi_4$ induced by $(\mu,\theta,\nu)$ given by $(x,y)$ is the flag matroid with bases $\{1,2,3,4,12,23,24,123,124\}$.  
This is also not a Bruhat polytope, because it contains the points $(2,3,1,4)$ and $(4,3,1,2)$ but not $(3,2,1,4)$ and $(4,2,1,3)$. 
\Cref{fig:13} (left) shows with red crosses these two points that do not correspond to Bruhat polytopes. 

If $y>x$, then the subdivision necessarily has one extreme in the interior of a vertical edge \Cref{fig:13}.
Hence the subdivision necessarily has one of these two non-Bruhat polytopes.
But if $x\le y$, it is possible to have a subdivision without these polytopes, by choosing $L_\theta$ to have its 0-dimensional cells in the horizontal sides of $R$.

The flag Dressian $\Fl\Dr((1,3),4)$ has a lineality space of dimension 3. 
If we mod out that lineality space, we get that $\Fl\Dr((1,3),4)$ is a cone over the complete graph $K_4$.
Each ray corresponds to a split over the hyperplane $\{x_i =0\}$. 
In \Cref{fig:FlDr} we see the the three different subsets of $\Fl\Dr((1,3),4)$ listed in \Cref{thm:hollow}.
The points at the end of the blue path in the left correspond to the rays $x_1=x_2\ge x_3,x_4$ and $x_3=x_4\ge x_1,x_2$.

\begin{figure}[h]
	\centering
		\includegraphics[width=\textwidth]{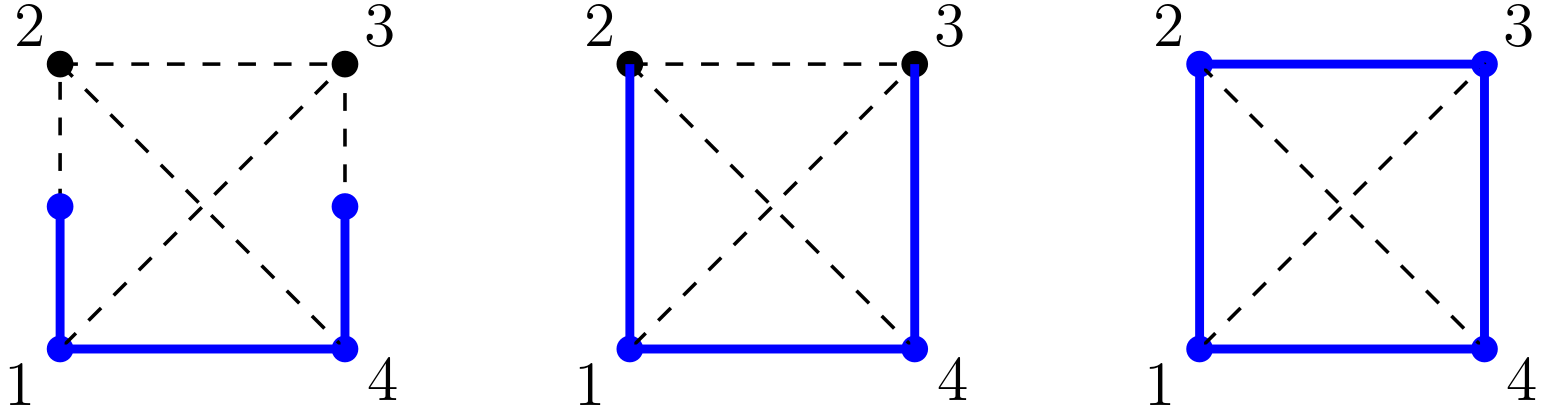}
	\caption{Left: The $\tnn$ tropical flag variety $\TFl^{\tnn}(1,3;4)$. Center: the subset of $\TFl(1,3;4)$ producing Bruhat subdivisions. Right: The non-negative flag Dressian $\FlDr^{\ge 0}(1,3;4)=\TFl^{\ge 0}(1,3;4)$.}
		\label{fig:FlDr}
\end{figure}

\end{example}
\begin{remark}
Notice that even though $\mu'$ and $\nu$ in the example above are valuated positroids and in fact $(\mu',\nu)$ is a valuated flag positroid, it can not be completed to a positive full flag. 
There are several positroids $\theta \in \Dr(2,4)$ such that $(\mu,\theta,\nu)$ is a complete flag.
However, none of them satisfy that $(\mu,\theta,\nu)$ is a flag positroid.
This example shows in particular that total positivity is not detected by the flag matroid subdivision.
\end{remark}

\section{Bruhat polytopes of hollow rank}
\label{sec:hollow_bruh}

In this section we look at Bruhat polytopes of rank $(1,n-1)$. 
We omit the superscript $(1,n-1)$ from $\BP_{u,v}^{(1,n-1)}$ and write simply $\BP_{u,v}$ as we only deal with hollow Bruhat polytopes for the next three sections.

\begin{remark}
Throughout this section, we use the untwisted convention for Bruhat polytopes, calling $v$ a minimal coset representative if it is the minimal elemnt in the fiber of $\pi^\fd\circ\phi$.
Similar proof probably exist when using the usual convention of twisted Bruhat polytopes, however we can choose the convention thanks to \Cref{thm:twisted}.
In fact, the desire to avoid translating all of this section to the twisted convention (the one used in the literature for non-complete flag varieties) was the main driving force to prove \Cref{thm:twisted}.
\end{remark}

Let us begin by taking care of polytopes in the boundary of $\Delta(1,n-1;n)$.

\begin{proposition}
Let $\fM=(M,N)$ be a hollow flag matroid over $[n]$. The flag matroid polytope $P_\fM$ is contained in the boundary of $\Delta(1,n-1;n)$ if and only if there is no element $i\in [n]$ which is neither a loop of $M$ or a coloop of $N$. 
\end{proposition}
\begin{proof}
Let $I\subset [n]$ be the set of non-loops of $M$ and $J$ the set of non-coloops of $N$.
If $I\cap J=\emptyset$, then $P_M$ is inside the hyperplane $\left\{\sum\limits_{i\in I} x_i = -1\right\}$.
Now if $I\cap J >0$, it must be have at least two elements, $i$ and $j$, because of (\ref{eq:troplucker}).
But then the origin is in the segment between the vertices $x$ and $y$ of $P_M$, where $x_i=y_j=1$ and $x_j=y_i=-1$ and every other coordinate equal to $0$.
So in particular $P_M$ intersects the interior of $\Delta(1,n-1;n)$
\end{proof}

As a consequence of the above proposition we have that every face of $\Delta(1,n-1;n)$ is completely contained inside one of the orthants of $\RR^n$.
Then by \Cref{prop:holgeo} non of the faces of $\Delta(1,n-1;n)$ are subdivided in matroid subdivisions.
Since $\Delta(1,n-1;n)$ is the Bruhat polytope of $[\id,(n,1)]$, we have that all flag matroid polytopes in the boundary of  $\Delta(1,n-1;n)$ are Bruhat by \Cref{prop:faces}.

Now we focus on hollow flag matroid polytopes intersecting the interior of $\Delta(1,n-1;n)$, or equivalently, containing the origin
\begin{definition}
Let $\fM$ be a hollow flag matroid whose polytope intersects the interior of $\Delta(1,n-1;n)$.
The \bemph{symbol sequence} $\alpha(\fM) = (\alpha_1,\dots,\alpha_n)\in\{0,+,-,*\}^n$ of $\fM$ is given by: 
\begin{itemize}
	\item $\alpha_i= *$ if both $i$ and $[n]\backslash i$ are bases of $\fM$
	\item $\alpha_i= -$ if $i$ is a base but $[n]\backslash i$ is not
	\item $\alpha_i= +$ if $[n]\backslash i$ is a base but $i$ is not
	\item $\alpha_i= 0$ if neither $[n]\backslash i$ nor $i$ are bases.
\end{itemize}
Given a sequence $\alpha \in \{0,+,-,*\}^n$ with at least two $*$, we write $P_\alpha$ for the polytope of that unique flag matroid with symbol $\alpha$.
\end{definition}

Equivalently, $\alpha_i= 0$ if $P_\fM$ is inside the coordinate hyperplane $\{x_i=0\}$, 
$\alpha_i=+$ if $P_\fM$ is in the positive half space $x_i\ge 0$, $\alpha_i=-$ if $P_\fM$ is in the negative half space $x_i\le 0$, and $\alpha_i = *$ if $P_\fM$ is in both sides of the hyperplane.
Such symbol sequences satisfy that at least two symbols are $*$ because of (\ref{eq:troplucker}), but you can get any sequence in $\{+,-,*\}^n$ with at least two $*$.

The polytope $P_\alpha$ is the convex hull all vertices of $\Delta(1,n-1;n)$ where the $-1$ is in coordinate $i$ where $\alpha_i \in \{-,*\}$ and the $+1$ is in a coordinate $j$ where $\alpha_j \in \{+,*\}$.
The dimension of $P_\alpha$ is $n-1-m$ where $m$ is the number of $0$'s in $\alpha$.

Consider a matroid subdivision $S_\fvm$ of a hollow flag matroid polytope $P_\fM$ intersecting the interior of $\Delta(1,n-1;n)$. 
The faces of a matroid subdivision $S_\fvm$ of $\Delta(1,n-1;n)$ in the interior correspond to the symbol sequences satisfying that $\alpha_i = *$ for all $i$ where $\lambda_i$ attains the minimum and no other $i$.
So there are $3^k$ of these polytopes where $k$ is the number of finite $\lambda_i$ that do not attain the minimum.
Moreover, the $f$-vector of the subdivision restricted to these faces is the mirror of the $f$-vector of a $k$-cube.


Given a permutation $v \in \Sym_n$, the corresponding point in $\Delta(1,n-1;n)$ is the one where $-1$ is at position $v^{-1}(1)$, $1$ is at position $v^{-1}(n)$.
So the minimal elements of the fiber $\pi^d\circ \phi$ are those permutations where $\{2,\dots,n-1\}$ are in ascending order.
If $v$ is a minimal element of its fiber $\pi^d\circ \phi$ and satisfies that $v^{-1}(n) < v^{-1}(1)$, a minimal word $w$ for $v$ can be constructed by an ascending subword starting with $\tau_1$ followed by a descending subword starting on $\tau_{n-1}$.
Explicitly, 
\begin{equation}
w=(\tau_1,\dots ,\tau_{k-2},\tau_{n-1},\tau_{n-2},\dots, \tau_m,)
\label{eq:word}
\end{equation}
where $k=v^{-1}(1)$ and  $m= v^{-1}(n)$.
%
If $k= 1$ then the ascending subword is empty and if $m=n$ then the descending subword is empty. 
Given a subword $w'$ of $w$, the corresponding vertex in $\Delta(1,n-1;n)$ has the $-1$ in position $k'$ where $\tau_1\dots \tau_{k'-1}$ is the largest ascending subword of $w'$ that starts with $\tau_1$ and the $1$ is in position $m'$ where $\tau_{n-1}\dots \tau_{m'}$ is the longest descending subword that starts with $\tau_{n-1}$.


%

Given a symbol sequence $\alpha$, we say it has an \bemph{isolated} $*$ if there exists $i$ such that $\alpha_i = *$ but neither $\alpha_{i-1}$ or $\alpha_{i+1}$ are $*$ 
(in the extremal cases we do not consider $\alpha_0$ or $\alpha_{n+1}$ to be $*$).
We are now ready to prove the main result of this section. 
\Cref{ex:thmbruh} showcases the constructions used in the proof.

\begin{theorem}
\label{thm:bruhat}
Let $\fM$ be a hollow flag matroid such that $P_\fM$ contains the origin. 
Then $P_\fM$ is a Bruhat polytope if and only if $\alpha(\fM)$ has no isolated $*$.
\end{theorem}


\begin{proof}

First, lets suppose that if a sequence $\alpha$ has an isolated $*$ and then $P_\alpha$ is not a Bruhat polytope.
Consider the face $F$ of $P_\alpha$ given by $x_i = 0$ for every $i$ such that $\alpha_i \ne *$. 
We show that this $F$ is not a Bruhat polytope, which implies that $P_\alpha$ is not Bruhat by \Cref{prop:richard}.  

We proceed by contradiction and suppose there are permutations $u \le v$ such that $F = \BP_{u,v}$. 
Let $w$ be the reduced word of $v$ as in (\ref{eq:word}) and consider $w'$ a subword of $w$ for $u$.
Let that isolated $*$ be $\alpha_i$.
Notice that we have that 
\[
v^{-1}(n) = m \le i \le v^{-1}(1) = k-1.
\] 

Lets suppose for now that $i$ is neither $1$ or $n$.
Since none of the vertices of $F$ have a $-1$ at the $i-1$ coordinate, any subword of $w$ for a permutation in $[u,v]$ with an ascending subword $\tau_1\dots \tau_{i-2}$ must have a $\tau_{i-1}$ afterwards. 
This shows that there must be a $\tau_{i-1}$ in $w'$.
Similarly, since none of the verticees of $\BP_{u,v}$ have a $1$ at the $i+1$ coordinate, there must be a $\tau_{i}$ in $w'$. 

Suppose that $\tau_{i-1}$ appears in $w'$ before $\tau_i$, so that $\tau_{i-1}\tau_i \le u$. 
In order to be able to get a vertex with $-1$ at coordinate $i$, we need to be able to get a subword $w''$ of $w$ with an ascending subword that stops at $\tau_{i-1}$. 
Then $w''$ needs to have a $\tau_{i-1}$ before a $\tau_i$ and another $\tau_{i-1}$ one after it. 
Moreover, $w''$ must have $\tau_{i-1}\tau_{i}\tau_{i-2}\tau_i$ as a subword.
However this is not a subword of $w$ which is a contradiction (if $i=2$ already the existence of $\tau_{i-1}\tau_i$ as a subword of $w'$ make it impossible to have a $-1$ in the second coordinate).

Now suppose that $\tau_i\tau_{i-1} \le u$. 
The argument is similar: to have a vertex with $1$ at coordinate $i$, we need a subword $w''$ of $w$ with a descending subword stopping at $\tau_i$.
But then this implies that $\tau_i \tau_{i-1} \tau_{i+1}\tau_i$ is subword of $w''$ but again this is not a subword of $w$. 

For the case $i=1$, to be able to have vertices of $P_M$ with $1$ at the first coordinate we need $w$ to have $\tau_n\dots \tau_1$ as a subword.
But there can not be vertices with $1$ at the second coordinate, so we can not allow descending subwords stopping at $\tau_2$.
Then we must have that $\tau_1$ is in $w'$. 
But then this rules out the possibility of having a $-1$ in the first coordinate, which contradicts that $\alpha_1 = *$.
A similar argument takes care of the case $i=n$.


Now to prove the other direction, we assume that $\alpha = \alpha(\fM)$ has no isolated $*$ and construct $u,v$ such that $\BP_{u,v} = P_\fM$. 
By \Cref{prop:richard}, it is enough to consider full dimensional polytopes, that is, $\alpha \in \{+,-,*\}^n$.
The choice of $v$ is straightforward, it is given by the word 
\[
w = (\tau_1,\dots ,\tau_{k-2},\tau_{n-1},\tau_{n-2},\dots ,\tau_m)
\]
where $k$ is the largest number such that $\alpha_k\in\{-,*\}$ and $m$ is the smallest number such that $\alpha_m \in \{+,*\}$.
We call $\tau_1,\dots ,\tau_{k-2}$ the ascending part of $w$ and $\tau_{n-1},\tau_{n-2},\dots ,\tau_m$ the descending part of $w$.
However, the choice of $u$ is more elaborate.


\begin{enumerate}
	\item First write in ascending order all $\tau_i$ for each $i$ such that 
\begin{itemize}
	\item $\alpha_{i+1} = +$ except when $\alpha_j = +$ for all $j\ge i+1$ or when $\alpha_j = -$ for all $j\le i$.
	\item $(\alpha_i,\alpha_{i+1}) = (+,*)$.
\end{itemize}
We call this the ascending part of $w'$.
\item Then write in descending order all $\tau_i$ for each $i$ such that 
\begin{itemize}
	\item $\alpha_{i+1} = -$ except when $\alpha_j = -$ for all $j< i$.
	\item $(\alpha_i,\alpha_{i+1}) = (-,*)$ except when $\alpha_j = -$ for all $j\le i$.
\end{itemize}
We call this the descending part of $w'$.
\end{enumerate}

It is straightforward to verify that the resulting word $w'$ is a subword of $w$; the fact that we exclude the cases where there are only $+$ after $\alpha_i$ in the ascending part and when there are only $-$ before $\alpha_i$ in the descending part ensures we never get out of the bounds given by $w$.
We choose $u$ to be the permutation given by $w'$.

Let us show that $\BP_{u,v}$ indeed gives us $P_\fM$.
First we show that all vertices in $P_\fM$ are in $\BP_{u,v}$.
Take a vertex $x$ of $P$ where $x_i=-1$ and $x_j=1$.
We will construct a permutation in $[u,v]$ that gives us $x$ with the following procedure:

\begin{enumerate}
	\item Start with $w'$.
	\item Fill all gaps in the ascending part until $\tau_{i-2}$ (so that the first $i-2$ letter are $\tau_{1}\dots\tau_{i-2}$).
	\item Fill all gaps in the descending part until $\tau_{j}$.
	\item If $j<i$ and $(\alpha_{i-1},\alpha_i) = (+,*)$, remove $\tau_{i-1}$ from the ascending part.
	\item If $(\alpha_{j-1},\alpha_j) = (-,*)$, remove $\tau_{j-1}$ from the descending part, except when $\alpha_l = -$ for all $l< j$.
	\item If $i<j$ $(\alpha_{j-1},\alpha_j) = (-,*)$, add $\tau_{j-1}$ to the ascending part, except when $\alpha_l = -$ for all $l< j$.
	\item If $i<j$ and $(\alpha_{i-1},\alpha_i) = (*,*)$, add $\tau_{i-1}$ to the descending part.
	\item If $i<j$ and $\alpha_l= -$, for all $l<i$, add $\tau_{i-1}$ to the ascending part.
\end{enumerate}
We call $w''$ the resulting word and $\sigma$ the corresponding permutation. 


Let us verify that $\sigma \in [u,v]$.
First we check $\sigma \le v$.
Filling gaps in steps 2 and 3 does not add letters to $w''$ that are not in $w$.
When adding the letter $\tau_{j-1}$ to the ascending part in step 6 we have that $\alpha_{j+1}=*$ so $k-2\ge j$.
When adding the letter $\tau_{i-1}$ to the descending part in step 7 we have $\alpha_{i-1}= *$ so $m\le i-1$. 
When adding the letter $\tau_{i-1}$ to the ascending part in step 8 we know that in $\alpha$ there are $*$ eventually so $k-2\ge i$. 
So all letters added are within the bounds given by $w$ and therefor $w''$ is a subword of $w$ and $\sigma\le v$.

Now lets check that $u\le \sigma$.
If $\tau_{i-1}$ was removed in step 4, then $(\alpha_i,\alpha_{i+1}) = (*,*)$ so $\tau_i$ is not a subword of $w'$.
As $\tau_{i-1}$ commutes with every $\tau_l$ for $l>i$, we can take $w'$ and move the $\tau_{i-1}$ in the ascending part after all $\tau_l$ for $l>i$ so that it is sitting in the corresponding place in the descending part without changing that $u$ is the resulting permutation.
Notice that $w''$ has a $\tau_{i-1}$ in the descending part because $j<i$ for this case.
Similarly, if $\tau_{j-1}$ was removed in step 5, the $(\alpha_j,\alpha_{j+1}) = (*,*)$ and there is no $\tau_j$ in $w'$ so we can move $\tau_{j-1}$ before all $\tau_l$ with $l>j$. 
We have that $\tau_{j-1}$ is in the ascending part of $w''$ because either $j>i$ and we added it since step 1, or $i<j$ and we added it in step 6.
This way we obtain a word for $u$ which is a subword of $w$ and $u\le \sigma$. 
Notice that here we are already using the fact that there are no isolated $*$.


Now we need to verify that $\sigma$ satisfies that $\sigma(i) = 1$ and  $\sigma(j) = n$.
To see that $\sigma(i) = 1$, we show that the largest ascending subword of $w''$ stops at $\tau_{i-1}$
By step 2, there is an initial ascending subword in $w''$ at least until $\tau_{i-2}$.
We have the following cases:
\begin{itemize}
	\item If $j<i$, then there is a $\tau_{i-1}$ in the descending part of $w''$ and this can not be followed by a $\tau_i$.
				The only problem that we could have is that there could already be a $\tau_{i-1}$ in the ascending part of $w'$.
				This can only happen if $\tau_{i-1}$ is already in the ascending part of $w'$. 
				But since $x_i = -1$, this is only possible if $(\alpha_{i-1},\alpha_i) = (+,*)$.
				However we removed this $\tau_{i-1}$ from $w''$ in step 4.
	\item If $i<j$ and $(\alpha_{i-1},\alpha_i) = (+,*)$, then there is a $\tau_{i-1}$ in the ascending part of $w''$.
				As there are no isolated $*$, we have that $(\alpha_i,\alpha_{i+1}) =(*,*)$, so there is no $\tau_i$ in $w''$.
	\item If $i<j$ and $(\alpha_{i-1},\alpha_i) = (*,*)$, then there are no $\tau_{i-1}$ in $w'$ but in step 7 there was a $\tau_{i-1}$ added to the descending part of $w''$.
				Since it is in the descending part it can not be followed by a $\tau_i$.
	\item If $i<j$ and $\alpha_l = -$ for all $l<i$, then $\tau_{i-1}$ was added to the ascending part in step 8. 
				In this case there is no $\tau_i$ in the ascending part of $w'$, because the only possibility would be for $(\alpha_i,\alpha_{i+1}) =(-,+)$ but this case is excluded when $\alpha_l = -$ for all $l<i$.
				Also because $\alpha_l = -$ for all $l<i$, $\tau_i$ was not added in step 6.
	\item The remaining cases are $i<j$, not all $\alpha_l = -$ for $l<i$ and either $\alpha_i = -$ or $(\alpha_{i-1},\alpha_i) = (-,*)$.
				Here we always have a $\tau_{i-1}$ in the descending part of $w'$ but not in the ascending part.
				This remains true in $w''$ so the ascending subword stops at $\tau_{i-1}$.
\end{itemize}

Fortunately, it is way easier to verify that $\sigma(j) = n$.
By construction we have a descending chain until $\tau_j$.  
Since $\alpha_j \in \{+,*\}$, the only case $\tau_{j-1}$ is in the descending part of $w'$ is when $(\alpha_{j-1},\alpha_j) = (-,*)$ and not all $\alpha_l = -$ for $l< j$, which is why it is removed from $w''$ in step 5.

Finally, we need to show that no other points appear in $\BP(u,v)$.
So we must show that for $\sigma \in [u,v]$ it can not happen that $\sigma(i)=1$ if $\alpha_i = +$ or $\sigma(i)=n$ if $\alpha_i = -$.
Let $w''$ be a word for $\sigma$.
Suppose $\alpha_i = +$ and the largest ascending subword stops at $\tau_{i-1}$, so that $\sigma(i)=1$. 
Notice that $\alpha_i = +$ implies that $w'$ has a $\tau_i$ as subword at the ascending or at the descending part.
So $w''$ also has $\tau_i$ as subword.
If $\alpha_l= -$ for all $l< i$, then there are no $\tau_l$ in the descending part of $w$ for $l<i$.
Therefor the ascending subword of $w''$ must all be in its ascending part and hence before the $\tau_i$.
But then we would have that $\sigma^{-1}(1)> i$, which is a contradiction. 
So we suppose that not all $\alpha_l= -$ for $l< i$.
Then there is a $\tau_{i-1}$ in the ascending part of $w'$ and hence $\tau_{i-1}\tau_i$ is a subword of $w''$. 
But then the final $\tau_{i-1}$ in the ascending subword of $w''$ must be after $\tau_i$ and therefor we have that $\tau_{i-1}\tau_i\tau_{i-2}\tau_i\le \sigma$.
However $\tau_{i-1}\tau_i\tau_{i-2}\tau_i \not \le v$, which is a contradiction.

Similarly, if $\alpha_i = -$ then for $\tau_{i}$ to exist in the descending part we have that not all $\alpha_l = -$ for $l<i $. 
So we have that $\tau_i \tau_{i-1}$ is a subword of $w''$ and if the largest descending subword ends at $\tau_i$ then $\tau_i\tau_{i-1}\tau_{i+1}\tau_i\le \sigma$ but $\tau_i\tau_{i-1}\tau_{i+1}\tau_i\not\le\sigma$.

\end{proof}

\begin{example}
\label{ex:thmbruh}
Consider the flag matroid $\fM =(M,N)$ over $[5]$ where $\BB(M) = \{1,3,4,5\}$ and $\BB(N)=\{1235,1245,1345\}$.
The symbol sequence of $\fM$ is $\alpha=(-,+,*,*,-)$. 
We get $w= \tau_1\tau_2\tau_3\tau_4\tau_3\tau_2$ and $w'= \tau_2\tau_4$, so $P_\fM=\BP_{u,v}$ for $u=(13254)$ and $v=(25341)$.
To verify that the point $x=(0,1,-1,0,0) \in \BP_{u,v}$ we follow the 8 steps to obtain $w''=\tau_1\tau_4\tau_3\tau_2$.
The resulting permutation is $\sigma = (25134)$ which indeed satisfies that $\sigma\in[u,v]$ and $x_{\sigma(1)}=-1$ and $x_{\sigma(5)}=1$.  
\end{example}

We get the second part of \Cref{thm:hollow} as corollary of \Cref{thm:bruhat}: 
\begin{theorem}
\label{thm:bruh}
A valuated flag matroid of hollow rank induces a Bruhat subdivision if and only if its lambda-values satisfy (\ref{eq:bruh}). 
\end{theorem}
\begin{proof}
Let $\fvm$ be a valuated flag matroid over $\fM$. 
We assume $P_\fM$ contains the origin as otherwise $S_\fvm$ is trivially a Bruhat subdivision and equation (\ref{eq:bruh}) is satisfied trivially as all $\lambda_i$ are infinite.
Since the polytopes we get in $S_\fvm$ have $\alpha_i= *$ if and only if $\lambda_i$ achieve the minimum, then if $\lambda_i$ achieves the minimum by \Cref{thm:bruhat} we have that either $\lambda_{i-1}$ or $\lambda_{i+1}$ achieve the minimum as well.
This is exactly what (\ref{eq:bruh}) says. 
\end{proof}

\section{Flag gammoids}
\label{sec:gammoid}
We now introduce a language which is useful to describe the paramtrization of $\tnn$ flag varieties given by Boretsky \cite{boretsky}.
This is based on the class of gammoids introduced by Mason in \cite{MasonGammoids} which was generalized in \cite{FO} to valuated matroids.
We adapt this construction to (valuated) flag matroids, an call them (valuated) flag gammoids.
For details on the classical theory of gammoids and the related construction of transversal matroids we refer to \cite{BrualdiWhite}.


Consider a directed graph $\Gamma = (V,E)$. 
We identify $n$ distinct vertices with $[n]$, so that $[n]\subseteq V$. 
Let $\fd=(d_1,\dots,d_s)$ be a rank vector and consider a flag $S_1\subset \dots \subset S_s$ of subsets of $V$ such that $|A_i|=d_i$ for every $i$. 
Let us call $\fS=(S_1,\dots,S_s)$ a flag of sinks, even though we do not require the vertices in $S_i$ to be sinks.
Given sets $I$ and $J$ of the same cardinality, a \bemph{linking }from $I$ onto $J$ is a collection of $|I|$ vertex-disjoint paths where each path starts at a vertex in $I$ and ends in a vertex in $J$.
The collection of subsets $I\subseteq [n]$ for which there is a linking onto $S_i$ form the bases of rank $d_i$ matroid (supposing this collection is not empty).
A matroid that can be constructed this way is called a \bemph{gammoid} \cite{MasonGammoids}. 
 
Now suppose each edge $e$ has a real weight $w_e$, that is, we have a function $w: E\to \RR$. 
For technical reasons, we assume there are no cycles with negative weight and there is always a subset of $[n]$ with a linking onto $S_i$ for every $i$.
The weight of a linking $\Phi$ is the sum of the weights of all edges in $\Phi$.
Given a subset $I\subset [n]$ of size $d_i$, let $\mu_i(I)$ be the minimum weight over all linkings from $I$ onto $S_i$.
Then $\mu_i$ is valuated matroid of rank $d_i$ known as \bemph{valuated gammoid} \cite{FO}.

\begin{theorem}
\label{thm:gammoid}
The valuated matroids $(\mu_1,\dots,\mu_s)$ defined as above form a valuated flag matroid $\mu(\Gamma,w,\fS)$ of rank $\fd$.
\end{theorem} 


\begin{proof}
First suppose $\fd=(1,\dots,s)$. 
We only need to prove that $\mu_i, \mu_{i+1}$ form a valuated matroid quotient.
To see this, consider giving the label $n+1$ to the unique vertex in $S_{i+1}\backslash S_i$. 
In the case that this vertex already had a label $k$, just create a new vertex labeled $n+1$ with a single edge to $k$ of weight $0$.
Let us call $\nu$ for the valuated matroid of rank $d_{i+1}$ over $[n+1]$ given by the linkings to $S_{i+1}$.  

Now, subsets of $I\subset [n+1]$ of size $d_{i+1}$ containing $n+1$ have a liking to $S_{i+1}$ if and only if $I\backslash(n+1)$ has a linking to $S_i$ with the same weight.
So we have that $\nu \backslash (n+1) = \mu_{i+1}$ and  $\nu / (n+1) = \mu_i$.
By \Cref{prop:bez} we have that $(\mu_i, \mu_{i+1})$ is a flag valuated matroid as desired.

Now for arbitrary rank $\fd$, extend the flag of sinks $\fS$ to a flag of sinks $(S_1',\dots, S_m')$ of rank $(1,\dots,m)$ where $m=|S_s|$.
This can be done by choosing a linear ordering to $S_s$ where if $x\in S_i$ and $y\in S_j\backslash S_i$ then $x<y$ and having $S_i'$ be the $i$ smallest elements of $S_s$.
Now we are in the previous case so we have a valuated flag matroid $\fvm'=(\mu_1',\dots,\mu_m')$.
We can then obtain $\fvm = (\mu_1,\dots,\mu_s)$ from $\fvm'$ by dropping all constituents whose rank is not in $\fd$, so $\fvm$ is a valuated flag matroid.
%
\end{proof}

\begin{definition}
\label{def:gammoid}
A (valuated) flag matroid $\mu=\mu(\Gamma,w,\fS)$ which arises from a directed graph $\Gamma$ with edge weights $w$ and a flag of sinks $\fS$ is called a \bemph{(valuated) flag gammoid}. 
\end{definition}

\begin{proposition}
\label{prop:gamma_trans}
Valuated flag gammoids are closed under translation.
\end{proposition}
\begin{proof}
Let $\fvm=\fvm(\Gamma, w,\fS)$ be valuated flag gammoid. 
Suppose you want to translate the flag of tropical linear spaces by an amount $r\in \RR$ in the $i$-th coordinate.
First suppose that $i$ is not a in the flag of sinks.
This corresponds to increasing $\mu_I$ by $r$ for every $I$ containing $i$.
To do that, consider $w'$ given by
\[
w'(e) := \begin{cases}
w(e) +r & \text{if } e \text{ is an outgoing edge from } i\\ 
w(e) -r & \text{if } e \text{ is an ingoing edge to } i\\
w(e)  & \text{otherwise. }
\end{cases}
\]
This is sometimes known as the \emph{gauge action} on $w$.
Equipping $\Gamma$ with these new edge weights has the effect that every path starting from $i$ is increased by $r$, but the weight of every other path remains the same, even if they use the vertex $i$ somewhere that is not the source.
The resulting flag matroid $\fvm'(\Gamma,w',\fS)$ is the translation of $\fvm$ we wanted.

If $i$ is in the the flag of sinks $\fS$, then just create a new vertex $v$ and an edge from $i$ to $v$ of weight $0$, replace $i$ by $v$ in $\fS$ and then apply the gauge transformation described above.
\end{proof}

When the vertices of $\Gamma$ are exactly $[n]$ with no other additional vertices, we can all the resulting (valuated) flag matroid a \bemph{(valuated) strict flag gammoid}, extending from the definitions in \cite{MasonGammoids}. Their duals can be called \bemph{(valuated) transversal flag matroid} \cite[Proposition 7.3]{FO}.
It is easy to prove that every polytope in the subdivision in a valuated (strict) flag gammoid is a (strict) gammoid using a suitable gauge transformation as described in \Cref{prop:gamma_trans} and removing edges with positive weight.
However the converse is far from trivial.
It was proven in \cite[Theorem 6.20]{FO} that a regular subdivision consisting of strict gammoids corresponds to a valuated strict gammoid.
We ask if we can generalize this in the flag setting, and moreover, to the non-strict setting:

\begin{question}
If we have a flag valuated matroid $\fvm$ such that every polytope in the subdivision is a (strict) gammoid, does it follow that $\fvm$ is a valuated (strict) flag gammoid? 
\end{question}
Notice that this question is similar in the local vs global spirit of some of the positivity questions we deal with in this paper. 
\Cref{ex:main} shows that a subdivision of $\tnn$-flag matroids does not imply that the valuated matroid is $\tnn$.
Whether a subdivision of flag positroids is a valuated flag positroid is still unknown to us; this is part of \Cref{conj:main}. \Cref{thm:arb}  provides a weaker version of this statement.  

Just like valuated positroids can be parametrized using valuated gammoids \cite{KodamaWilliams:2015}, we can use valuated flag gammoids to parametrize flag matroids in $\tnn$ tropical flag varieties. 
The following is an adaptation of a construction given in \cite{boretsky}, which itself is an interpretation of the parametrization of Richardson varieties given in \cite{MR}.
Our construction is an ad-hoc simplification for hollow flags using the language of flag gammoids and using \Cref{thm:bruh} to use as a starting point the flag matroid instead of the Bruhat interval.

Let $\fM= (M,N)$ be a matroid of rank $(1,n-1)$ such that $P_\fM$ is a Bruhat polytope.
Let $a_1<\dots<a_\alpha$ be the non-loops of $M$.
By the second part of \Cref{thm:bruh}, if there is an element $i$ which is neither a loop of $M$ nor a coloop in $N$, then either $i-1$ or $i+1$ satisfy the same.
So the set $C$, consistinig of the elements which are not loops of $M$ nor coloops of $N$ consists of consecutive segments of size at least $2$.
Let $b_1<\dots<b_\beta$ be the set of non-coloops of $N$ except for the last element of each consecutive segment in $C$.

We construct a weighted directed graph which we call $\Gamma_\fM$ in the following way.
The set of vertices is $[n]\sqcup[\overline{n}]$ where $[\overline{n}] =\{\overline{1},\dots,\overline{n}\}$. 
There are three kinds of edges:
\begin{itemize}
	\item For every $i$, there is an edge from $i$ to $\overline{i}$ with weight $w(i,\overline{i})=0$.
	\item For every $i<\alpha$, there is an edge from $a_i$ to $a_{i+1}$ with arbitrary weight  $w(a_i,a_{i+1})$.
	\item For every $i<\beta$, There is an edge from $b_i$ to $\overline{b_{i+1}}$ with arbitrary weight $w(b_i,\overline{b_{i+1}})$.
\end{itemize}
Let $S_1=\{\overline{a_\alpha}\}$ and $S_2=[\overline{n}]\backslash \{\overline{b_1}\}$ and consider the valuated flag matroid $\fvm(\Gamma_\fM, w,\fS)$ for the flag of sources $\fS=(S_1,S_2)$.

\begin{example}
\label{ex:graph}
Consider the flag matroid $\fM=(M,N)$ on $[8]$ where the loops of $M$ are $3$ and $8$ and the only coloop of $N$ is $7$. 
This matroid satisfies (\ref{eq:bruh}) so $P_\fM$ is a Bruhat polytope. 
The graph $\Gamma_\fM$ is pictured in \Cref{fig:graph}.

\begin{figure}[h]
	\centering
		\includegraphics[width=0.7\textwidth]{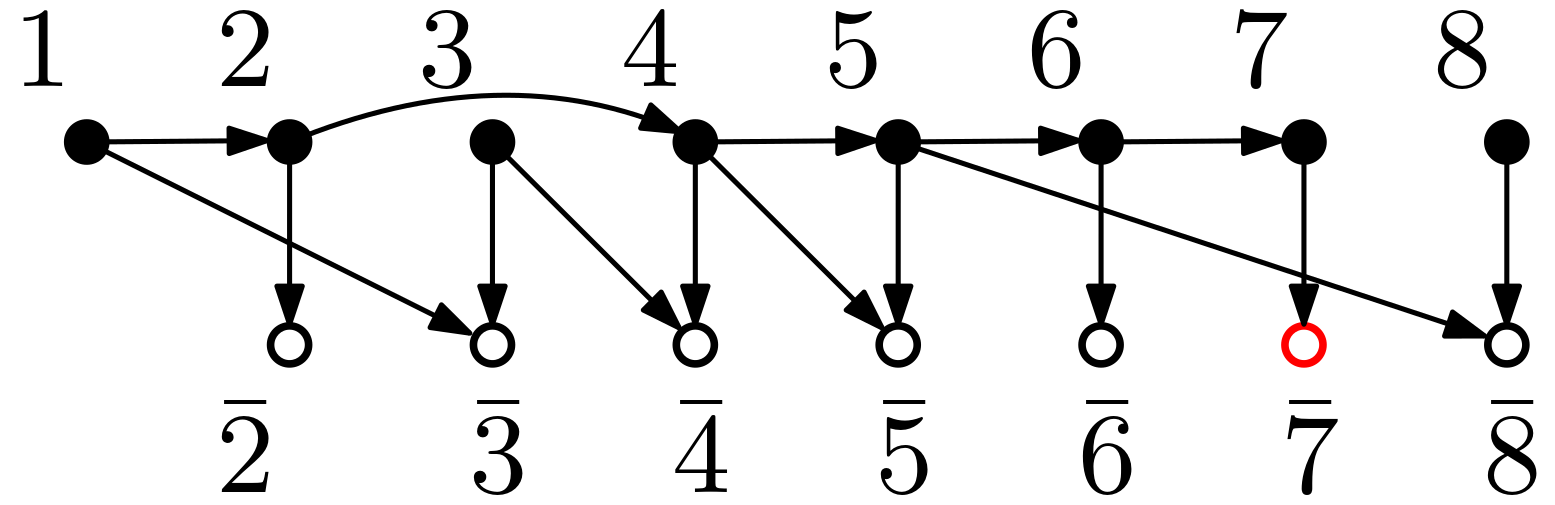}

	\caption{The graph $\Gamma_\fM$ of \Cref{ex:graph}. The sink $\bar{7}$ is in red as it is the only element in $S_1$. The element $\bar{1}$ is removed as it is the only element of $[\bar{n}]$ not in $S_2$, so it is irrelevant.}
	\label{fig:graph}
\end{figure}

\end{example}

\begin{remark}
The construction in \cite{boretsky} has several convention differences to this one.
First, the elements of the matroid are placed as sinks as and look at linkings onto them from marked sources. 
However in the classical construction of gammoids \cite{BrualdiWhite,MasonGammoids} the arrows go the other way, looking instead at likings from a subset of elements of a matroid onto marked sinks, which is what we use here.
Moreover, the twisted convention of Bruhat polytopes is used instead of the untwisted.
From the proof of \Cref{thm:twisted} we have that to change conventions one takes a Bruhat interval for the dual and then takes inverse of the permutations in the interval.
Notice that a word for the inverse of a permutation can be obtained by reversing the order of the letters.
So together the change of these conventions amounts to doing the same construction in \cite[Definition 3.11]{boretsky} for $[u',v']$ where $P_{\fM^*}= \BP_{u',v}$ while reversing all arrows and placing the arrows from right to left in the figures therein.
\end{remark}

\begin{proposition}
A valuated flag matroid $\fvm$ with support $\fM$ is in $\TFl^\tnn(1,n-1;n)$ if and only if $\fvm=\mu(\Gamma_\fM,w,\fS)$ for some edge weights $w$.  
\label{prop:param} 
\end{proposition}
\begin{proof}
If $P_{\fM^*}=\BP_{u,v}$ for some Bruhat interval $[u,v]$ of rank $(1,n-1)$, then $v$ has a word which consists of an ascending subword $\tau_1\tau_2\dots$ followed by a descending subword $\tau_{n-1}\tau_{n-2}\dots$. 
Following the construction in \cite[Definition 3.11]{boretsky}, we have that the ascending subword corresponds to some ascending path $a_1 \rightarrow \dots \rightarrow a_\alpha$ in the vertices labeled $[n]$ and the descending part corresponds to edges $(b_i,\overline{b_{i+1}})$ for some sequence $b_1< \dots< b_\beta$.
We need to show that the sets $A=\{a_1,\dots,a_\alpha\}$ and $B=\{b_1,\dots b_\beta\}$ are the same as the ones constructed above for $\Gamma_\fM$.

First, it is clear that the only vertices with a path to $S_1=\{\overline{a_\alpha}\}$ are those in $A$. 
So necessarily $A$ must be the set of bases i.e. non-loops of $M$.  
Then for every element in $b_i$ we can obtain a linking from $[n]\backslash b_i$ onto $S_2$ by taking the edges $(b_{j-1},\overline{b_j})$ for $j\le i$ and $(j,\overline{j})$ for all other remaining sinks in $S_2$. 
However, if $b_i$ is also in $A$, say $b_i = a_j$, then we can also add the edge $(a_j,a_{j+1})$ and obtain a linking from $[n]\backslash a_{j+1}$ onto $S_2$, so $b_i$ must not be a coloop in $N$.
By \Cref{thm:bruh}, there is always at least two consecutive elements which are both non-loops of $M$ and non-coloops of $N$.
Hence we can and must omit the last element in every consecutive segment of non-coloops of $N$ which are also not loops of $M$.
By taking a minimal $u$ we can assume every other non-coloop of $N$ is in $B$, showing that indeed $\Gamma_\fM$ is the end result.
The statement follows by \cite[Theorem 3.13]{boretsky}.
\end{proof}

\section{Positivity for hollow flag matroids}
\label{sec:hollow+}
In this section we will prove part 1. (\Cref{thm:tnnh}) and part 3. (\Cref{thm:posh}) of \Cref{thm:hollow}.
We begin with the following lemma:

\begin{lemma}
\label{lemma:posequ}
A sequence of real numbers $\lambda_1,\dots, \lambda_n$ satisfies (\ref{eq:tnnh}) if and only if there exists $x_1,\dots,x_{n-1}$ such that $\lambda_i = x_i\oplus x_{i-1}$ for every $i$ (where $x_0=x_n=\infty$). 
\end{lemma}
\begin{proof}
First, if such $x_1,\dots x_{n-1}$ exist, then
\[
\lambda_{i-1}\oplus\lambda_{i} = (x_{i-1}\oplus x_{i-2})\oplus(x_{i+1}\oplus x_i) \le x_{i-1}\oplus x_{i} = \lambda_i.
\]

Now suppose that $\lambda_1,\dots, \lambda_n$ satisfy (\ref{eq:tnnh}). 
Let $x_i := \max(\lambda_i,\lambda_{i+1})$.
We have that $x_i\oplus x_{i-1}= \max(\lambda_i,\lambda_{i+1}) \oplus \max(\lambda_{i+1},\lambda_i)$ is larger than $\lambda_i$ because both terms are larger than $x_i$. 
But for both terms to be larger than $x_i$, we would have that $\lambda_{i+1}\oplus \lambda_{i-1} > \lambda_i$, which contradicts (\ref{eq:tnnh}).
\end{proof}

\begin{theorem}
\label{thm:tnnh}
A valuated flag matroid of hollow rank is $\tnn$ if and only if its lambda-values satisfy (\ref{eq:tnnh}). 
\end{theorem}
\begin{proof}
Let $\fvm=(\mu,\nu)$ be a valuated matroid satisfying (\ref{eq:tnnh}).
Then its support $\fM=(M,N)$ satisfies (\ref{eq:bruh}). 
By \Cref{prop:param}, $\fvm$ is in $\TFl(1,n-1;n)$ if and only if it can be represented as a valuated flag gammoid by $\Gamma_\fM$ as constructed in \Cref{sec:gammoid}.

First, as total non-negativity is invariant under translation, we can translate $\fvm$ so that $\mu$ is a $(0,\infty)$-vector. 
This has the effect that the edges of the form $(a_i,a_{i+1})$ must be all of weight $0$, since $\mu_{a_j}$ is the sum of all these edges for $i\ge j$. 
Furthermore, by scaling $\nu\in \TP$ we can assume that $\nu_{b_1}=0$.
This way $\lambda_i= \nu_{[n]\backslash i}$ is the minimum weight of all linkings from $[n]\backslash i$ to $S_2$.

Let $y_i$ be the weight of the edge $(b_i,\overline{b_{i+1}})$.
Let $x_i=y_1+\dots + y_{i-1}$, so in particular $x_1 = 0$.
Clearly we can recover the edge weights $y_i$ from the $x_i$ and viceversa. 
Let us consider the possible linkings of $[n]\backslash i $ into $S_2$. 
\begin{enumerate}[I.]
	\item If $i=b_j$ for some $j$, then we can have the path $(b_k,\overline{b_{k+1}})$ for every $k<j$ and $(m,\overline{m})$ for every sink $\overline{m}\in S_2$ not yet used.  
This linking has weight $x_j$.
	\item If $i=a_l$ for some $l>2$ and $a_{l-1}=b_j$ for every some $j$, then we can have the path $(a_{l-1},a_l,\overline{a_l})$, the path $(b_k,\overline{b_{k+1}})$ for every $k< j$ and $(m,\overline{m})$ for every sink $\overline{m}\in S_2$ not yet used.
Again, this linking has weight $x_j$. 
\end{enumerate}
No other linking exists. 
We call these linkings type I and type II respectively.

If $b_i$ is a loop of $M$, then we have that $[n]\backslash b_i$ has only linking onto $S_2$ of type I and weight $x_i$.
Now consider a consecutive segment $b_i,b_i+1\dots,b_i+k$ of $C$, the set of non-loops of $M$ and non-coloops of $N$.
In other words $\lambda_{b_i},\dots,\lambda_{b_i+k}$ are finite but $\lambda_{j-1}=\lambda_{k+1}=\infty$.
Equivalently, $b_i+j=b_{i+j}$ for $0\le j<k$ but not for $j=k$ or $j=-1$. 
We compute $\lambda_{b_i+j}$ considering the following cases
\begin{itemize}
	\item Because $b_i$ is the first element of a consecutive segment of $C$, the largest non-loop of $M$ before $b_i$ is either a coloop in $N$ or the last element in another consecutive segment of $C$. Either way, we have that $b_i$ does not have a linking of type II, it only has a linking of type I with weight $x_i$, so $\lambda_{b_i}=x_i$.
	\item As $b_i+k$ is the last element of the cosecutive segment, it is not equal to some $b_m$, so it can not have a linking of type I.
It only has a linking of type II which has weight $x_{i+k-1}$, so $\lambda_{b_i+k}=x_{i+k-1}$.
	\item If $1\le j <k$, then $b_i+j$ has a linking of type I of weight $x_i$ and linking of weight $x_{i+j}$ and a linking of type II of weight $x_{i+j-1}$, so $\lambda_{b_i+k}=x_{i+j}\oplus x_{i+k-1}$.
\end{itemize}
By \Cref{lemma:posequ}, $\lambda_{b_i},\dots,\lambda_{b_i+k}$ must satisfy (\ref{eq:tnnh}).
If every consecutive segment of finite lambda values satisfy (\ref{eq:tnnh}), then all of them satisfy (\ref{eq:tnnh}).
Conversely, if all $\lambda_1,\dots,\lambda_n$ satisfy (\ref{eq:tnnh}), then again we can choose appropriate $x_i$ on each consecutive segment as in the proof of \Cref{lemma:posequ}.  
Namely, $x_i=\max(\lambda_{b_i},\lambda_{b_i+1})$ if both $\lambda_{b_i}$ and $\lambda_{b_i+1}$ are finite and $b_i+1$ is not the last element of the consecutive segment in $C$.
If $b_i$ is the last element of a consecutive segment of $C$ or it is a loop of $M$, then $x_i= \nu_{[n]\backslash b_{i+1}}$.
This way we can recover $x_i$ and hence recover the edge weights $y_i$ which will give us the valuated matroid $\fvm$ by \Cref{lemma:posequ}. 
\end{proof}

\begin{example}
\label{ex:totally}
Recall the flag matroid $\fM$ from \Cref{ex:graph}.
Consider $\fvm$ be the valuated flag matroid with support in $\fM$ where the valuation on the bases of $M$ are all equal to $0$ and 
\begin{align*}
\mu_{[n]\backslash 1} = 0 \quad & \mu_{[n]\backslash 2} = 0 \quad \mu_{[n]\backslash 3} = 2 \quad  \mu_{[n]\backslash 4} = 1\\
\mu_{[n]\backslash 5} = 1 \quad & \mu_{[n]\backslash 6} = 3 \quad \mu_{[n]\backslash 7} = \infty \quad  \mu_{[n]\backslash 8} = -1.
\end{align*}
Hence we have that $\lambda=(0:0:\infty:1:1:3:\infty:\infty)\in \TT\PP^7$ satisfies (\ref{eq:tnnh}).
Recall the graph $\Gamma_\fM$ pictured in \Cref{fig:graph}.
As the valuation on $M$ is trivial, we can assume all horizontal edges have weight 0 (as well as the vertical weights).
Thus we just need to determine the weights of the diagonal edges.
Using the procedure described above, we get the edge weights
\[
y_{1,\bar{3}} = 2 \quad y_{3,\bar{4}} = -1 \quad  y_{4,\bar{5}} = 2 \quad  y_{5,\bar{8}} = -4.
\]
If we had a valuated flag matroid $\fvm$ where the valuation on $M$ is not trivial, then we can translate to a valuated flag matroid $\fvm'$ where the valuation in $M$ is trivial to obtain edge weights as above and then use the gauge action as described in \Cref{prop:gamma_trans} to obtain a graphical presentation of $\fvm$ as a totally positive valuated flag matroid. 
\end{example}


\begin{theorem}\label{thm:posh}
A valuated flag matroid $\fvm$ of hollow rank is a flag positroid if and only if $\fvm\in \FlDr^{\ge 0}(1,n-1;n).$
\end{theorem} 
\begin{proof}
One of the directions is immediate for arbitrary rank, so we do the other direction.
Suppose $\lambda$ satisfies (\ref{eq:posh}) for hollow valuated flag matroid $\fvm(\mu,\nu)$.
Let $i$ be an even index where $\lambda$ achieves the minimum and $j$ an odd index where $\lambda$ achieves the minimum.
As the lambda-values, as well as realizability by positive linear spaces is invariant under translation, we suppose again that $\mu$ is the trivial valuation on its support.
Moreover, we can take a representative of $\nu$ that has non-negative entries and that $\nu([n]\backslash i)=\nu([n]\backslash j)=0$.

Let us build a matrix $A\in \RR\ldb t \rdb^{(n-1)\times n}$ that realizes $\fvm$ as a flag positroid.
For the first row of $A$ let
\[
A_{1,k} = \begin{cases} n \quad \text{if $k=i$ or $k=j$}\\
0 \quad \text{if $k$ is a loop of $M$}\\
1 \quad \text{otherwise.}
\end{cases}
\]
Let $A_{f(k),k}=1$ for every $k\in [n]\backslash\{i,j\}$ where $f$ is the order preserving function $[n]\backslash\{i,j\} \to [n-1]\backslash 1$.
Let $A_{k,i}=\alpha t^{\nu([n]\backslash k)}$ for every even $k$ and $A_{k,j}=\alpha t^{\nu([n]\backslash k)}$ for every odd $k$ where
\[
\alpha = \begin{cases} 1 \quad \text{if $j<i$}\\
-1\quad \text{if $i<j$}
\end{cases}
\]
(where $t^\infty := 0$).
Finally, let every other entry of $A$ equal to $0$.
See (\ref{eq:matrix}) for an example of such a matrix.

Let us show that the row spans of $A$ realizes $\fvm$ as flag positroid.
Clearly the first row has non-negative coordinates and realizes $\mu$.
Now let us look at the $(n-1)\times (n-1)$ minors.
Consider the minor given by deleting the column $k$ with $k$ different from $i$ or $j$.
Suppose $k$ is even.
Any function $\sigma: [n]\backslash k  \to [n-1]$ on non-zero entries must have that $\sigma(i)=k$ as it is the only column with a non-zero entry in that row.
Now $\sigma(j)$ could be $1$ or any odd value.
If $\sigma(j)= 1$, then for rest of the coordinates we have $\sigma(l) = f(l)$. 
Because of how $\alpha$ was defined, the resulting sign for $\sigma$ is going to be positive and the product of the corresponding entries is going to be equal to $nt^{\nu([n]\backslash k)}$.
If $\sigma(j)\ne 1$, the valuation of the resulting product is at least $\nu([n]\backslash k)$. 
This is because we would need $\sigma(j)$ not to be a loop of $M$ for the entry $A_{1,\sigma(j)}$ to be different to $0$ and then $\nu([n]\backslash \sigma(j))=\lambda_{\sigma(j)}$ is non-negative. 
However, the absolute value of the leading coefficient of such term would be $1$ 
So the term  $nt^{\nu([n]\backslash k)}$ dominates all others.
Hence this minor is positive and has valuation $\nu([n]\backslash k)$ as desired.

A similar reasoning shows that this is also the case for $k$ odd.
If $k=i$, then again the dominating term in the minor corresponds to $\sigma(j)=1$ and $\sigma(l)=f(l)$ and this term is equal to $n$.
Similarly for $k=j$, so this matrix realizes $\fvm$ as a flag positroid.
\end{proof}

Notice that valuated flag positroids always induce flag positroids subdivisions.
On the other hand if a hollow valuated flag matroid $\fvm$ does not satisfy (\ref{eq:posh}), then clearly none of maximal cells in the subdivision induced by $\fvm$ are going to be flag positroid polytopes.
Therefor \Cref{thm:posh} implies part 3. of \Cref{thm:hollow}.
With this we complete our proof of \Cref{thm:hollow} and we have now completely determined all positivity notions with explicit conditions on the Pl\"ucker coordinates for the hollow case.

\begin{example}
Recall the flag valuated matroid from \Cref{ex:totally}. 
In this case we have $i=2$ and $j=1$ and the matrix

\begin{equation}
A= \begin{pmatrix}
	8 & 8 & 0 & 1 & 1 & 1 & 1 & 0\\
	0 & t^2 & 1 & 0 & 0 & 0 & 0 & 0\\
	t^1 & 0 & 0 & 1 & 0 & 0 & 0 & 0\\
	0 & t^1 & 0 & 0 & 1 & 0 & 0 & 0\\
	t^3 & 0 & 0 & 0 & 0 & 1 & 0 & 0\\
	0 & 0 & 0 & 0 & 0 & 0 & 1 & 0\\
	t^{-1} & 0 & 0 & 0 & 0 & 0 & 0 & 1
	\end{pmatrix} 
\label{eq:matrix}.
\end{equation}
realizes it as a flag positroid.
This matrix may not have all minors positive, though it is possible to construct such a matrix by recovering the parametrization of \cite{MR} from the weighted graph constructed in \Cref{ex:totally}.
\end{example}

\section{Arbitrary rank}
\label{sec:arb}

In this section we turn our attention to arbitrary rank $\fd$ and use \Cref{thm:hollow} to prove \Cref{thm:arb}.
Before we continue to the proof, we want to comment on why we have to include Pl\"ucker pairs $(S,T)$ corresponding to Pl\"ucker relations, i.e. when $S\in {[n] \choose d_i-1}$ and  $T\in {[n] \choose d_i+1}$.

In \cite[Theorem 3.10]{JLLO} it is shown that the incidence relations 
on some functions $\mu:{[n]\choose d}\to \RR$ and $\nu:{[n]\choose d+1}\to \RR$ imply the Pl\"ucker relations on $\mu$ and $\nu$. 
We could ask whether the same holds for valuated matroids whose rank differ by more than $1$.
However, the answer is no, as the following example shows:
\begin{example}
Consider the $\mu:{[5]\choose 2}: \RR$ where $\mu(ij)=0$ for $ij=12,23,34,45,15$, and an arbitrary positive number in the rest of the cases.
This function satisfies all incidence relations with $\nu=(0,0,0,0,0)\in \Dr(4,5)$.
However $\mu$ is not a valuated matroid; for example, the minimum in $\mu(12)\odot\mu(34)\oplus \mu(13)\odot\mu(24) \oplus \mu(14)\odot\mu(24)$ is achieved only in $\mu(12)\odot\mu(34)$.
\end{example}

Furthermore the positive version of the three-term incidence relations imply the positive Pl\"ucker relations for the complete flag variety of full support \cite[Lemma 6.7]{JLLO}. 
We could ask whether this extends to partial flags, adding the assumption that the functions are actually valuated matroids, because of the previous example.
Again, the answer is no. 

\begin{example}
Let $\fvm\in \Dr(2,4;5)$ where $\fvm(13)= 1$ and $\fvm(I)=0$ for every other set $I\in {5\choose 2}\cup{5\choose 4}$.
This valuated flag matroid satisfies the positive relations ((\ref{eq:posa}) for $S\in {5\choose 1}$ and $T=[5]$), but $\mu_1$ does not satisfy the positive Pl\"ucker relations as
$\fvm(13)\odot\fvm(24)>\fvm(12)\odot\fvm(34)\oplus\fvm(14)\odot\fvm(23)$.
\end{example}

Because the positive three-term Pl\"ucker relations are equivalent to a valuated matroid being a valuated positroid \cite{arkani2017positive,SpeyerWilliams:2021}, all of the positivity concepts lisited in the introduction imply that the constituents of $\fvm$ are valuated positroids.
So from now on we will assume that $\fvm$ has constituents that are valuated positroids and we will only consider Pl\"ucker pairs $(S,T)$ that correspond to incidence relations (i.e. $T\in {[n]\choose d_{i+1}+1}$)

\begin{definition}
\label{def:hp}
Let $(S,T)$ be a Pl\"ucker pair of $\fvm$. 
Then the \bemph{hollow projection} of $\fvm$ induced by $(S,T)$ is the valuated flag matroid 
\[
\eta_{S,T}(\fvm) := (\mu_i/S|_T, \mu_{i+1}/(S\cap T)|_T)
\] 
\end{definition}
The adjective hollow is to emphasize that the result is a matroid over $T\backslash S$ of hollow rank.
To see that it is actually a flag matroid, notice that $(\mu_i/S|_T, \mu_{i+1}/S|_T,\mu_{i+1}/(S\cap T)|_T)$ is a flag matroid, the first quotient is due to the fact that taking minors preserves valuated matroid quotient \cite[Corollary 4.3.2]{BEZ} and the second quotient is due to \Cref{prop:bez}. 

The reason to consider hollow projections is the following:

\begin{lemma}
\label{lemma:eta}
The lambda values of $\eta_{S,T}(\fvm)$ equal the lambda values of $\fvm$ for the Pl\"ucker pair $(S,T)$.
\end{lemma}
\begin{proof}
Let $(\mu',\nu') = \eta_{S,T}(\fvm)$. 
The lemma simply follows from the fact that for all $j\in T\backslash S$ we have that $\mu'(j)=\fvm(Sj)$ and $\nu'(T\backslash Sj) = \fvm(T\backslash j)$.
\end{proof}

\Cref{thm:arb} follows directly from \Cref{thm:hollow}, \Cref{lemma:eta} and the following theorem:

\begin{theorem}
Let $(S,T)$ be a Pl\"ucker pair of flag valuated matroid $\fvm$ where $|S|=d_i-1$ and $|T|=d_{i+1}+1$.
\begin{enumerate}
	\item If $\fvm$ is $\tnn$, then $\eta_{S,T}(\fvm)$ is $\tnn$.
	\item If $\fvm$ induces a Bruhat subdivision, then $\eta_{S,T}(\fvm)$ induces a Bruhat subdivision.
	\item If $\fvm$ induces a positroid subdivision, then $\eta_{S,T}(\fvm)$ induces a positroid subdivision.
\end{enumerate}
\end{theorem}

\begin{proof}
\begin{enumerate}
\item Suppose $\fvm$ is totally non negative. So there is a flag of linear spaces $\fL=(L_1,\dots,L_n)\in \Fl^\tnn(\fn,C^n)$ which restricted to $\fd$ realizes $\fvm$ as a $\tnn$ valuated flag matroid. 
Consider $\pi_T$ the projection to the $T$ coordinates and 
\[
L_j'= \pi_T(L_j\cap\{x_k=0 \mid k\in S\}).
\]
We have that $(L_{d_i}',\dots, L_{d_{i+1}}')$ is a flag of linear spaces of consecutive rank $(1,\dots,|T|-|S|-1)$ with positive Pl\"ucker coordinates.
Now order the elements of $S\backslash T$, say $a_1,\dots,a_k$ and let 
\[L_j''= \pi_T(L_{d_{i+1}}\cap\{x_k=0 \mid k\in S\backslash\{a_l \mid l< j\}\}).\] 
We obtain another flag of linear spaces $(L_1'',\dots,L_{k+1}'')$ of linear spaces of consecutive rank $(|T|-|S|-1,\dots, |T\backslash S|-1)$. 
Notice that $L_1'' =  L_{d_{i+1}}'$, $L_1'$ is a realization of $\mu'$ and $L_{k+1}''$ is a realization of $\nu'$.
So 
\[
(L_{d_i}',\dots, L_{d_{i+1}}'= L_1'',\dots,L_{k+1}'')
\]
is a complete flag with all of its Pl\"ucker coordinates positive showing that $\eta_{S,T}(\fvm)$ is totally non-negative.

	\item Suppose $\fvm$ induces a Bruhat subdivision. 
	Let $x\in \RR^{T\backslash S}$ and $(M,N):= \eta_{S,T}(\fvm)^x$.
	

Let $\tilde{x}\in \RR^n$ be a vector such that $\tilde{x}|_{T\backslash S} = x$, $x_i$ is sufficiently large for $i\in S$ and sufficiently small for $i\in [n]\backslash (S\cup T)$. 
For such $\tilde{x}$, we have that $\mu_{d_i}^{\tilde{x}}/S|_T = M$ and $\mu_{d_i+1}^{\tilde{x}}/S|_T = (\mu_{d_i+1}/S|_T)^x$. 
As $\fvm^{\tilde{x}}$ is a Bruhat polytope, so is $\fvm^{\tilde{x}}/S|_T$.
In particular we have that $(M, (\mu_{d_i+1}/S|_T)^x)$ can be completed to flag positroid $\fM'=(M_1',\dots,M_{|T|-|S|-1}')$ of consecutive rank.
In particular, $\fM'$ satisfies the 3-term incidence relations.

Let $\mu_j'' = \mu_{d+1}/(S\backslash\{a_l \mid l< j\})|_T$ where again $S\backslash T = \{a_1,\dots a_k\}$.
If $\fvm$ induces a Bruhat subdivision then $\mu_{d+1}$ induces a positroid subdivision and hence satisfies the 3-term positive Pl\"ucker relations.  
This implies that $\fvm'':=(\mu_1'',\dots,\mu_{k+1}'')$ and therefor $\fM''=\fvm''^x =(\mu_1''^x,\dots,\mu_{k+1}''^x)$ satisfy the 3-term positive incidence relations.

So the flag matroid $(M_1',\dots,M_{|T|-|S|-1}' = \mu_1''^x, \dots, \mu_{k+1}''^x)$ satisfies all the 3-term positive incidence relations and by \cite[Theorem 6.10]{JLLO} it has a Bruhat polyotpe.
But $M_1'=M$ and $\mu_{k+1}''^x = N$, so the restriction of this flag matroid to hollow rank is $(M,N)$. 
So $P_{(M,N)}$ is a Bruhat polytope. 
As $x\in \RR^{T\backslash S}$ was arbitrary, we have that $\eta_{S,T}(\fvm)$ induces a Bruhat subdivision.

\item This proof is similar to the one above.
Suppose $\fvm$ induces a positroid subdivision.
Let $x\in \RR^{T\backslash S}$ and $(M,N):= \eta_{S,T}(\fvm)^x$.
Let $\tilde{x}$ be as before. We have that $(\mu_{d_i}^{\tilde{x}},\mu_{d_i+1}^{\tilde{x}})= (M, \mu_{d_i+1}^{\tilde{x}}/S|_T)$ is a flag positroid and it satisfies the positive incidence relations (\ref{eq:posa}).
Let $\fM''$ be as before. We have that $\fM''$ is a flag positroid, so it also satisfies the positive incidence relations.
So we have that $(M,\mu_{d_i+1}^{\tilde{x}}/S|_T= \mu_1''^x,\dots,\mu_{k+1}''^x=N)$ is the non negative flag Dressian by \cite[Theorem 2.16]{JO}, (see \cite[Remark 3.11]{BEW}).
This implies that $(M,N)\in \FlDr^{\ge0}(1,|T\backslash  S|-1,n)$ and, by \Cref{thm:posh}, it is a flag positroid.
As $x$ was arbitrary, $\eta_{S,T}(\fvm)$ induces a positroid subdivision. 
 
\end{enumerate}
\end{proof}

Now we discuss some information that should be taken into account regarding \Cref{conj:main}.

\begin{proposition}
\label{prop:unvaluated}
\begin{enumerate}
	\item If part 2. of \Cref{conj:main} holds for unvaluated matroids, then it holds for all valuated matroids.
	\item If \ref{enum:dress} implies \ref{enum:positroid} for unvaluated matroids, then \ref{enum:dress} implies \ref{enum:positroid} holds for all valuated matroids.
\end{enumerate}

\end{proposition}
\begin{proof}
Recall that the lambda values are invariant under translation. 
So for every $x\in \RR^n$ and for each Pl\"ucker pair $(S,T)$, the lambda values of the flag matroid $\fvm^x$ for $(S,T)$ is either the result of replacinig by $0$ every term that achieves the minimum and everything else is replaced by $\infty$, or they all become $\infty$. 
Either way, if $\fvm$ satisfies (\ref{eq:brua}) then  $\fvm^x$ satisfies (\ref{eq:brua}) and if $\fvm$ satisfies (\ref{eq:posa}) then  $\fvm^x$ satisfies (\ref{eq:posa}).
%
\end{proof}

In other words, part 2. of \Cref{conj:main} follows from a generalization of \Cref{thm:bruhat}, while \ref{enum:dress} $\leftrightarrow$ \ref{enum:positroid} is equivalent to \cite[Question 1.6]{BEW}.
Notice that parts 1. and 2. of \Cref{conj:main} are the same at the level of matroids. Hence we have the following

\begin{corollary}
\label{coro:conjimp}
Part 1. of \Cref{conj:main} implies part 2. 
\end{corollary}

%
%
%

\begin{remark}
It may be tempting to believe that to prove \Cref{conj:main} it is enough for flags with 2 constituents, as one could then fill the gaps independently.
However, if one we have that $(\mu_1,\mu_2)$ and $(\mu_2,\mu_3)$ can be realized by flags $(L_1,L_2)$ and $(L_2',L_3)$ with positive Pl\"ucker coordinates, they might not satisfy that $L_2=L_2'$.
\end{remark}


We finish with the following questions:

\begin{question}
If we have a regular Bruhat subdivision, does there always exist a $\tnn$ flag valuated matroid that induces it? 
\end{question}

\begin{question}
Is the (positive) flag Dressian generated by (positive) Pl\"ucker-incidence relations of minmal length (i.e. three term Pl\"ucker relations plus incidence relations with $S\subseteq T$)?
\end{question}


\small
\bibliography{bruhat.bib}
\bibliographystyle{alpha}

\end{document}